\documentclass[12pt,twoside,leqno]{article}
\usepackage{amsthm,amsfonts,amssymb,amsmath,amscd,enumerate}
\usepackage[colorlinks=true]{hyperref}
\usepackage[mathscr]{eucal}
\usepackage[all]{xy}
\usepackage{mathtools}
\usepackage{mathrsfs}
\usepackage{comment}
\usepackage{listings} 
\theoremstyle{plain}

\newcommand{\Addresses}{{
  \bigskip
  \footnotesize

  \textsc{Fakultät für Mathematik, Universität Regensburg, 93040 Regensburg, Germany}\par\nopagebreak
  \textit{E-mail address}: \texttt{yanshuai.qin@ur.de}\par 
\textsc{IMJ-PRG, Sorbonne Universit\'e, 4 place Jussieu, 75005 Paris, France}\par\nopagebreak
  \textit{E-mail address}: \texttt{zhenghuili@imj-prg.fr}\par\nopagebreak
  \textsc{Instytut Matematyczny Polskiej Akademii Nauk
ul. Śniadeckich 800-656 Warszawa}\par\nopagebreak
  \textit{E-mail address}: \texttt{vertl@impan.pl}\par 

}}

\newcommand{\CA}{{\mathcal {A}}}

\newcommand{\CO}{{\mathcal {O}}}

\newcommand{\CX}{{\mathcal {X}}}
\newcommand{\CY}{{\mathcal {Y}}}

\newcommand{\Char}{{\mathrm{char}}}
\newcommand{\Br}{{\mathrm{Br}}}

\newcommand{\red}{{\mathrm{red}}}

\newcommand{\Gal}{{\mathrm{Gal}}}
\newcommand{\GL}{{\mathrm{GL}}}

\newcommand{\Hom}{{\mathrm{Hom}}}

\newcommand{\Ker}{{\mathrm{Ker}}}

\newcommand{\NS}{{\mathrm{NS}}}

\newcommand{\non}{{\mathrm{non}\text{-}}}

\newcommand{\Pic}{\mathrm{Pic}}

\renewcommand{\Im}{{\mathrm{Im}}}

\font\cyr=wncyr10  \newcommand{\Sha}{\hbox{\cyr X}}

\newcommand{\Spf}{{\mathrm{Spf}}}

\newcommand{\syn}{{\mathrm{syn}}}
\newcommand{\tor}{{\mathrm{tor}}}

\newcommand{\nr}{{\mathrm{nr}}}

\newcommand{\fppf}{\mathrm{fppf}} 
\newcommand{\et}{\mathrm{et}} 
 
\newcommand{\rig}{\mathrm{rig}}
\newcommand{\cris}{\mathrm{crys}}

\DeclareMathOperator{\Spec}{Spec}

\newcommand{\QQ}{\mathbb{Q}}

\newcommand{\ZZ}{\mathbb{Z}} 
\newcommand{\FF}{\mathbb{F}} 
\newcommand{\GG}{\mathbb{G}}

\newcommand{\lra}{\longrightarrow}

\newcommand{\SI}{{\mathscr{I}}}

\newtheorem{thm}{Theorem}[section]
\newtheorem{cor}[thm]{Corollary}
\newtheorem{lem}[thm]{Lemma}
\newtheorem{prop}[thm]{Proposition}
\newtheorem{conj}[thm]{Conjecture}
\newtheorem{defn}[thm]{Definition}

\theoremstyle{definition}

\newtheorem{example}[thm]{Example}

\theoremstyle{remark}
\newtheorem{rem}[thm]{Remark}

\def\bC{{\mathbb C}}

\def\bF{{\mathbb F}}
\def\bG{{\mathbb G}}

\def\bN{{\mathbb N}}

\def\bP{{\mathbb P}}
\def\bQ{{\mathbb Q}}

\def\bZ{{\mathbb Z}}

\def\cA{{\mathcal A}}
\def\cB{{\mathcal B}}

\def\cD{{\mathcal D}}
\def\cE{{\mathcal E}}
\def\cF{{\mathcal F}}
\def\cG{{\mathcal G}}
\def\cH{{\mathcal H}}
\def\cI{{\mathcal I}}
\def\cJ{{\mathcal J}}

\def\cL{{\mathcal L}}

\def\cO{{\mathcal O}}
\def\cP{{\mathcal P}}

\def\cX{{\mathcal X}}
\def\cY{{\mathcal Y}}

\def\sE{{\mathscr{E}}}
\def\sF{{\mathscr{F}}}

\def\sO{{\mathscr{O}}}
\def\sP{{\mathscr{P}}}

\def\sX{{\mathscr{X}}}
\def\sY{{\mathscr{Y}}}

\def\fQ{{\mathfrak{Q}}}

\def\fY{{\mathfrak{Y}}}

\def\frQ{{\mathfrak{Q}}}

\def\frT{{\mathfrak{T}}}

\def\frY{{\mathfrak{Y}}}

\DeclareMathOperator{\Fisoc}{F-Isoc}
\DeclareMathOperator{\Fisocd}{F-Isoc^{\dagger}}
\DeclareMathOperator{\Fcris}{F-Crys}
\DeclareMathOperator{\crys}{crys}
\DeclareMathOperator{\Crys}{Crys}
\DeclareMathOperator{\CRYS}{CRYS}
\DeclareMathOperator{\SYN}{SYN}
\DeclareMathOperator{\ZAR}{ZAR}
\DeclareMathOperator{\Ogus}{Ogus}
\DeclareMathOperator{\isoc}{Isoc}
\DeclareMathOperator{\isocd}{Isoc^{\dagger}}
\DeclareMathOperator{\conv}{conv}
\DeclareMathOperator{\Spa}{Spa}
\DeclareMathOperator{\Mod}{Mod}
\DeclareMathOperator{\ssp}{sp}
\DeclareMathOperator{\ffp}{fp}

\DeclareMathOperator{\coker}{coker}

\begin{document}

\renewcommand{\thefootnote}{\fnsymbol{footnote}} 
\footnotetext{\emph{Key words}:  Brauer group, Tate conjecture, Tate-Shafarevich group}   
\footnotetext{\emph{MSC classes}: 11G40, 14G17, 14J20, 11G25, 11G35}
\footnotetext{The paper is written when the first named author is in his Ph.D program in Sorbonne Universit\'e. This program has received funding from the European Union's Horizon 2020 research and innovation programme under the Marie Skłodowska-Curie grant agreement No 945332 and partially funded by the project “Group schemes, root systems, and related representations” founded by the European Union - NextGenerationEU through Romania’s National Recovery and Resilience Plan (PNRR) call no. PNRR-III-C9-2023- I8, Project CF159/31.07.2023.  The second named author is supported by the DFG through CRC1085 Higher Invariants (University of Regensburg).}
\title{On $p$-torsions of  geometric Brauer groups }
\author{Zhenghui Li, Yanshuai Qin \\ \\ \textit{With an appendix by Veronika Ertl}}

\maketitle

\begin{abstract}
Let $X$ be a smooth projective integral variety over a finitely generated field $k$ of characteristic $p>0$. We show that the finiteness of the exponent of the $p$-primary part of $\Br(X_{k^s})^{G_k}$ is equivalent to the Tate conjecture for divisors, generalizing D'Addezio's theorem for abelian varieties to arbitrary smooth projective varieties. In combination with the Leray spectral sequence for rigid cohomology derived from the Berthelot conjecture recently proved by Ertl-Vezzani, we show that the cokernel of $\Br_{\nr}(K(X)) \rightarrow \Br(X_{k^s})^{G_k}$ is of finite exponent. This completes the $p$-primary part of the generalization of Artin-Grothendieck's theorem on relations between Brauer groups and Tate-Shafarevich groups to higher relative dimensions.
\end{abstract}
\section{Introduction}
For any regular noetherian scheme $X$,  the \emph{cohomological Brauer group} $\Br(X)$ is defined to be the \'etale cohomology group $H^2_{\et}(X,\GG_m)$.
Let $K$ be a finitely generated field. The \emph{unramified Brauer group} $\Br_{\nr}(K)$ is defined as the intersection of Brauer groups of all DVRs with fraction field $K$. By the purity of Brauer groups \cite{Ces}, if $K$ admits a proper regular model $X$, then $\Br_{\nr}(K)=\Br(X)$. Recall the Tate conjecture for divisors: 
\begin{conj}[$T^1(X,\ell)$]
 Let $X$ be a smooth projective geometrically integral variety over a finitely generated field $k$ and $\ell$ be a prime number not equal to the characteristic of $k$, the cycle class map 
$$ \Pic(X)\otimes_\ZZ\QQ_\ell \lra H^2_\et(X_{k^s},\QQ_\ell(1))^{G_k}$$
is surjective.    
\end{conj}
It is well-known that $T^1(X,\ell)$ is equivalent to the finiteness of $\Br(X_{k^s})^{G_k}[\ell^\infty]$. In the case that $k$ is of characteristic $p>0$, a natural question is that how to formulate the conjecture using $p$-adic cohomology? This was done by P\'al \cite{Pal} using rigid cohomology, and he proved that his $p$-adic Tate conjecture is equivalent to the $\ell$-adic Tate conjecture $T^1(X,\ell)$. A natural question that arises is whether the 
$p$-torsion elements of the geometric Brauer group $\Br(X_{k^s})^{G_k}$ are the obstruction of P\'al's $p$-adic Tate conjecture for divisors? If $X$ is an abelian variety, the question was solved by D'Addezio \cite{DAd} by showing that $\Br(X_{k^s})^{G_k}[p^\infty]$ has finite exponent (the Tate conjecture for abelian varieties was known by Zarhin's work \cite{Zar1,Zar2}). To address the question in full generality, we establish the following theorem:
\begin{thm}\label{mainthm}
Let $k$ be finitely generated field of characteristic $p>0$. Let $X/k$ be a smooth projective geometrically integral variety. The finiteness of the exponent of $\Br(X_{k^s})^{G_k}[p^\infty]$ is equivalent to the Tate conjecture for divisors on $X$. The natural map $\Br_{\nr}(K(X)) \rightarrow \Br(X_{k^s})^{G_k}$ has a cokernel of finite exponent. 
\end{thm}
\begin{cor}\label{cor1.3}
Let $X$ be an abelian variety or a K3 surface or any smooth projective variety satisfying the Tate conjecture for divisors, defined over a finitely generated field $k$ of characteristic $p>0$. Then, the geometric Brauer group $\Br(X_{k^s})^{G_k}$ has finite exponent. Moreover, it is finite if the Witt-vector cohomology $H^2(X_{\bar{k}}, W\CO_{X_{\bar{k}}})$ is finitely generated over $W(\bar{k})$ and $\Pic_{X/k}$ is smooth over $k$. 
\end{cor}
The prime-to-$p$ part of the corollary was proved by Skorobogatov and Zarhin \cite{SZ1,SZ2} for abelian varieties and K3 surfaces and by Cadoret-Hui-Tamagawa \cite{CHT} in general. The $p$-primary part for abelian varieties was proved by D'Addezio \cite{DAd} by a different method. The p-primary part of the corollary for K3 surfaces, independent of us, is also obtained by Lazda and Skorobogatov \cite{2024boundednesspprimarytorsionbrauer} using the Kuga–Satake construction.

For an abelian group $M$  and a prime $\ell$, denote by $V_\ell M$ the rational Tate-module $\varprojlim_{n}M[\ell^n]\otimes_{\ZZ_\ell}\QQ_\ell$. For an abelian variety $A$ over a finitely generated field $k$, define the Tate-Shafarevich group $\Sha(A)$ as the kernel of the natural map
$$H^1(k,A) \lra \prod_{v\in M_k} H^1(k_v^{sh}, A),$$
where $M_k$ denotes the set of discrete valuations on $k$ (cf. \cite{EKQ} for more details).
\begin{cor}\label{Artin-Grothendieck}
Let $X$ be a smooth projective geometrically integral variety over a finitely generated field $k$ and $\ell$ be a prime number. Then there exists a canonical exact sequence
$$
0\longrightarrow V_\ell\Br_{\nr}(k)\longrightarrow \Ker(V_\ell\Br_{\nr}(K(X))\longrightarrow V_\ell\Br(X_{k^{s}})^{G_k}) \longrightarrow V_\ell\Sha(\Pic^0_{X/k,\red})\longrightarrow 0.
$$
Moreover, the natural map
$$\Br_{\nr}(K(X))\rightarrow \Br(X_{k^{s}})^{G_k}$$
has a cokernel of finite exponent.
\end{cor}
This result is a higher dimensional generalization of Artin-Grothendieck's theorem \cite[\S \, 4]{Gro3} on the relation between Brauer groups and Tate-Shafarevich groups. The prime-to-$p$ part ($p=\Char(k)$) of the corollary was proved by the second named author \cite{Qin1, Qin2} for $k$ being a global field and by Ertl-Keller-Qin \cite{EKQ} in general. The first claim of the corollary was proved by  Ertl-Keller-Qin \cite{EKQ}.
\begin{cor}\label{fgfield}
Let $\CX \lra \CY$ be a projective surjective $k$-morphism between smooth projective varieties over a finitely generated field $k$ of characteristic $p>0$. Assume that its generic fiber $X/K$ is smooth and geometrically integral. The cokernel of the natural map
$$\Br(\CX_{k^s})^{G_k}\lra \Br(X_{K^s})^{G_K}$$
has finite exponent.\\
\end{cor}

In \cite{DAd}, D'Addezio defined $H^2_{\fppf}(X_{\bar{k}},\bQ_p(1))^k$ as the image of 
$$H^2_{\fppf}(X,\bQ_p(1))\longrightarrow H^2_{\fppf}(X_{\overline{k}},\bQ_p(1)),$$
and he gave a formulation of the $p$-adic Tate conjecture for divisors:
\begin{conj}[D'Addezio]
The cycle class map 
$$\Pic(X)\otimes_\bZ\bQ_p \lra H^2_{\fppf}(X_{\bar{k}},\bQ_p(1))^k$$ is surjective.
\end{conj}
We prove the following equivalence via using the pull-back trick developed by Colliot-Th\'el\`ene and Skorobogatov (cf. \cite{CTS1}), which was used in D'Addezio's proof of his proposition 3.9 in \cite{DAd}.
\begin{thm}
The above conjecture is equivalent to the finiteness of the exponent of $\Br(X_{k^s})^{G_k}[p^\infty]$.
\end{thm}
Using Theorem \ref{mainthm} and a Lefschetz pencil argument developed by Ambrosi in \cite[\S 4]{ambrosi2021specializationneronseverigroupspositive}, we obtain a weak Lefschetz type result for Brauer groups.
\begin{thm}\label{mainlef}
Let $Y$ be a smooth projective geometrically integral variety over a finitely generated field $k$ in positive characteristic. Assume that $k$ is an infinite field and $\dim(Y)\geq 3$. Then, there exist infinitely many 
 smooth hyperplane sections $Z\subseteq Y$ such that
the pullback morphism
$$\Br(Y) \lra \Br(Z)$$
has a kernel and a cokernel of finite exponent.
\end{thm}
\begin{rem}
This question was studied by Xinyi Yuan in an unpublished note \cite{Yua3}. For a finitely generated field $k$ of characteristic $0$, the same method shows that the claim holds for $n^\infty$-torsions of the Brauer groups, where $n$ is any positive integer. Assuming the Tate conjecture for divisor, any smooth hyperplane section satisfies the condition in the theorem, as geometric Brauer groups are of finite exponent under this assumption. If one drops the finitely generated condition for $k$, the claim may not hold. For example, let $Y=\bP_{\bC}^3 \subseteq \bP_{\bC}^{34}$ be the Veronese embedding. In this case, any smooth hyperplane section $Z$ is a K3 surface with an infinite Brauer group.
\end{rem}

\subsection*{Outline of Proofs}
Firstly, by spreading out $X/k$, we get a smooth projective morphism $f:\CX \lra U$ between smooth integral varieties over $\FF_p$. On the one hand, P\'al reformulated the Tate conjecture using rigid cohomology \cite[\S\, 6]{Pal}\footnote{In P\'al's original formulation, it is not clear to us how to simultaneously establish the Leray spectral sequence and identify higher direct image with Ogus' one after completion. Since Berthelot's conjecture (of the structure sheaf) is recently proved by \cite{ertl2024berthelotsconjecturehomotopytheory}, the easiest way to fix it is to replace higher direct images in \cite{Pal} by Berthelot's higher direct image (see Appendix \ref{appendixB}). However, we also present another way in Appendix \ref{appendixA}, following P\'al, via convergent cohomology. So the first claim of our main theorem \ref{mainthm} only depends on a weaker version of Berthelot's conjecture proved by Ambrosi \cite{ambrosi2021specializationneronseverigroupspositive} while the second claim still needs the stronger version.}: the natural map
$$
\NS(X)\otimes_\ZZ\QQ_p\lra H^0_{\rig}(U, R^2f_*\CO_{\CX}^\dag)^{F=p}
$$
is surjective. On the other hand, there is a canonical exact sequence (cf. Proposition \ref{firstexact})
$$
0\lra\NS(X)\otimes_\ZZ\QQ_p\lra \varprojlim_{n}H^0_{\fppf}(\Spec(k), R^2f_{*}\mu_{p^n})\otimes_{\ZZ_p}\QQ_p \lra V_p\Br(X_{k^s})^{G_k} \lra 0.
$$
It remains to prove
$$
 \varprojlim_{n}H^0_{\fppf}(\Spec(k), R^2f_{*}\mu_{p^n})\otimes_{\ZZ_p}\QQ_p \cong H^0_{\rig}(U, R^2f_*\CO_{\CX}^\dag)^{F=p}.
$$
Secondly, we use the syntomic sheaf $S_n^1$ defined by Bauer in \cite{Bauer1992} which is isomorphic to the sheaf $\mu_{p^n}$ on the small syntomic site of $X$ to relate the flat cohomology to crystalline cohomology. By a deep result of Fontaine, there is an exact sequence of sheaves on the small syntomic site \cite[Thm. 3.4]{Bauer1992}
$$0\lra \mu_{p^n}\lra \SI_n \stackrel{1-p^{-1}F}{\lra} \CO_{n}^{\cris}\lra 0.$$
It gives a natural map
$$
H^0_{\syn}(\Spec(k), R^2f_{\syn,*}\mu_{p^n}) \lra H^0_{\syn}(\Spec(k), R^2f_{\syn,*}\SI_n)^{p^{-1}F=1}.
$$
The second term is essentially isomorphic to $H^0_{\cris}(\Spec(k), R^2f_{\cris,*}\CO_{X/\Sigma_n})^{F=p}$ up to groups of finite exponent (bounded independent of $n$). To show that the natural map is an isomorphism up to groups of finite exponent (bounded independent of $n$), we use a pull-back trick developed by Colliot-Th\'el\`ene and Skorobogatov in \cite{CTS1}. This is divided into three steps. First, we show that the claim is true for curves by a direct computation. Second, we can take the Albanese variety of $X$ and show that the natural morphism between $R^1f_*$ induced by the Albanese map is an isomorphism up to group of finite exponent. Third, we use the pull-back trick to show that the natural map above for general $X$ is an isomorphism up to groups of finite exponent (bounded independent of $n$).
\subsection*{Notations}
\subsubsection*{Fields}
By a \emph{finitely generated field}, we mean a field which is finitely generated over a prime field.\\
For any field $k$, denote by $k^s$ the separable closure. Denote by $G_k=\Gal(k^s/k)$ the absolute Galois group of $k$.
\subsubsection*{Varieties}
By a \emph{variety} over a field $k$, we mean a scheme which is separated and of finite type over $k$. A morphism between varieties over $k$ is a $k$-morphism.

For a smooth proper geometrically connected variety $X$ over a field $k$, we use $\Pic^0_{X/k}$ to denote the identity component of the Picard scheme $\Pic_{X/k}$. Denote by $\Pic^0_{X/k,\red}$ the underlying reduced closed subscheme of $\Pic^0_{X/k}$ (which is known to be an abelian variety cf. \cite[p236-p17, Cor. 3.2.]{FGA} ). The N\'eron-Severi group $\NS(X)$ of $X$ is defined as the quotient $\Pic(X)/\Pic^0(X)$.
\subsubsection*{Cohomology}
The default sheaves and cohomology over schemes are with respect to the
small \'etale site. So $H^i$ is the abbreviation of $H_{\et}^i$. $H^i_{\fppf}$ denotes the cohomology of sheaves on the $\fppf$ site.
\subsubsection*{Brauer groups}
For any noetherian scheme $X$, denote the \emph{cohomological Brauer group}
$$
\Br(X):=H^2(X,\mathbb{G}_m)_{\tor}.
$$
\subsubsection*{Abelian group}
For any abelian group $M$, integer $m$ and prime $\ell$, we set\\
$$M[m]=\{x\in M| mx=0\},\quad M_{\tor}=\bigcup\limits_{m\geq 1}M[m],\quad  M[\ell^\infty]=\bigcup\limits_{n\geq 1}M[\ell^n], $$
$$T_\ell M=\Hom_\mathbb{Z}(\mathbb{Q}_\ell/\mathbb{Z}_\ell, M)=\lim \limits_{\substack{\leftarrow \\ n}}M[\ell^n],\quad V_\ell M= T_\ell(M)\otimes_{\mathbb{Z}_\ell}\mathbb{Q}_\ell.$$
For an abelian group $M$ and a prime number $p$, we set
$$M(\non p)=\bigcup\limits_{p\nmid m}M[m].$$ 
Let $f:M\rightarrow N$ be a morphism in an abelian category,  we say that $f$ is \emph{al. injective} (resp. \emph{al. surjective} ) if the kernel (resp. cokernel ) can be killed by a positive integer. We say that $f$ is an \emph{al. isomorphism} iff it is al. injective and al. surjective. If $f_n$ are a family of morphisms in an abelian category indexed by a set $I$, we say $f$ is \emph{al. injective  uniformly in $n$} (resp. \emph{al. surjective  uniformly in $n$}) if the kernels (resp. cokernels ) can be killed by the same positive integer simultaneously. We say that $f$ is an \emph{al. isomorphism uniformly in $n$} iff it is al. injective uniformly in $n$ and al. surjective uniformly in $n$. For an object or a morphism ( resp. a family of them ) in an abelian category, we say that it is al. zero ( resp. al. zero uniformly in $n$ ) if it can be killed by some positive integer (resp. by the same positive integer).

\subsubsection*{Sites and Topos}
For a scheme $X$, we let $\syn(X)$ (resp. $\SYN(X)$) be the small (big) syntomic site and let $\fppf(X)$ be the big fppf site.\\

For a continuous functor $f:T\to T'$ between sites, we can define an adjoint pairing $(f_s,f^s)$ between their topos. The spectral sequences for $f^s$ (cf. \cite[Theorem 3.7.5; 3.7.6]{tamme2012introduction}) exists even if $f_s$ is not exact. We still denote $f^s$ by $f_*$ and denote $f_s$ by $f^{-1}$.\\

For the crystalline site and topos, we follow the notation in \cite{BerthelotOgus1978}. Let $S=(S,I,\gamma)$ be a PD-scheme and $X$ be an $S$-scheme to which $\gamma$ extends, we use $\Crys(X/S)$ to denote the small crystalline(-Zariski) site and $(X/S)_{\crys}$ to denote the topos. If $S$ is clear, we may omit it and write it as $X_{crys}$. For $\tau\in\{\ZAR,\SYN\}$, we use $\CRYS(X/S)_\tau$ to denote the big crystalline site equipped with $\tau$-topology and $(X/S)_{\CRYS,\tau}$ to denote the corresponding topos.

\section{Preliminaries}
Let $k$ be a field. Let $f:X\lra S=\Spec(k)$ be a smooth projective geometrically integral variety. Let $S^\prime$ denote $\Spec(k^s)$. The goal of this section is to prove the following proposition
\begin{prop}\label{firstexact}
There is a canonical exact sequence
$$
0\lra\NS(X)\otimes_\ZZ\QQ_p\lra \varprojlim_{n}H^0_{\fppf}(S, R^2f_{*}\mu_{p^n})\otimes_{\ZZ_p}\QQ_p \lra V_p\Br(X_{k^s})^{G_k} \lra 0,
$$
where $R^2f_{*}\mu_{p^n}$ is the higher direct image of the $\fppf$ sheaf $\mu_{p^n}$ on $X$.
\end{prop}
\begin{proof}
Consider the Leray spectral sequence for the $\fppf$ sheaf $\mu_{p^n}$ on $X$
$$
E_2^{p,q}=H^p_{\fppf}(S, R^qf_{*}\mu_{p^n}) \Rightarrow H^{p+q}_{\fppf}(X, \mu_{p^n}).
$$
It gives an edge map
$$ H^{2}_{\fppf}(X, \mu_{p^n})\lra H^0_{\fppf}(S, R^2f_{*}\mu_{p^n}).$$
Consider the Leray spectral sequence for the $\fppf$ sheaf $\GG_m$ on $X_{k^s}$
$$
E_2^{p,q}=H^p_{\fppf}(S^\prime, R^qf_{*}\GG_m) \Rightarrow H^{p+q}_{\fppf}(X_{k^s}, \GG_m).
$$
Since $f_*\GG_m=\GG_m$, $E^{p,0}_2=H^p(S^\prime, \GG_m)=0$ for $p\geq 1$. There is an exact sequence 
$$
0 \lra H^1_{\fppf}(S^\prime, R^1f_{*}\GG_m) \lra H^{2}_{\fppf}(X_{k^s}, \GG_m) \lra H^0_{\fppf}(S^\prime, R^2f_{*}\GG_m) \lra H^2_{\fppf}(S^\prime, R^1f_{*}\GG_m).
$$
By Lemma \ref{3vanishing} below,
$$ \Br(X_{k^s})=H^{2}_{\fppf}(X_{k^s}, \GG_m) \lra H^0_{\fppf}(S^\prime, R^2f_{*}\GG_m)$$
is surjective and has a kernel of finite exponent. Consider the commutative diagram
\begin{displaymath}
\xymatrix{                         
H^{2}_{\fppf}(X, \mu_{p^n})\ar[r]\ar[d]& H^0_{\fppf}(S, R^2f_{*}\mu_{p^n})\ar[d] 
\\
\Br(X_{k^s})^{G_k}[p^n] \ar[r] & H^0_{\fppf}(S^\prime, R^2f_{*}\GG_m)^{G_k}[p^n]
}
\end{displaymath}
The second row is an al. isomorphism uniformly in $n$. By Lemma \ref{bigtrick} below, the first column is al. surjective uniformly in $n$. Thus, the second column is al. surjective uniformly in $n$.  The Kummer exact sequence on the $\fppf$ site of $X$ 
$$ 0\lra \mu_p^n \lra \GG_m\stackrel{p^n}{\lra}\GG_m\lra 0$$
gives an exact sequence of $\fppf$ sheaves
$$ 0\lra R^1f_*\GG_m/p^n \lra R^2f_*\mu_p^n \lra  R^2f_*\GG_m.$$
Hence we get
$$ 0\lra H^0_{\fppf}(S,R^1f_*\GG_m/p^n) \lra H^0_{\fppf}(S, R^2f_*\mu_p^n) \lra  H^0_{\fppf}(S, R^2f_*\GG_m).$$
By \cite[Lem. 3.3]{DAd}, the natural map $H^0_{\fppf}(S, R^2f_{*}\GG_m) \lra H^0_{\fppf}(S^\prime, R^2f_{*}\GG_m)^{G_k} $ is injective. This gives an exact sequence with the last map al. surjective uniformly in $n$
$$ 0\lra H^0_{\fppf}(S,R^1f_*\GG_m/p^n) \lra H^0_{\fppf}(S, R^2f_*\mu_p^n) \lra  H^0_{\fppf}(S^\prime, R^2f_*\GG_m)^{G_k}[p^n].$$
By Lemma \ref{nscompu} below,  $H^0_{\fppf}(S,R^1f_*\GG_m/p^n)$ is al. isomorphic to $(\NS(X_{\bar{k}})/p^n)^{G_k}$ uniformly in $n$. Taking inverse limit and then tensoring $\QQ_p$, we get 
$$
0\lra(\NS(X_{\bar{k}})\otimes_\ZZ\QQ_p)^{G_k}\lra \varprojlim_{n}H^0_{\fppf}(S, R^2f_{*}\mu_{p^n})\otimes_{\ZZ_p}\QQ_p \lra V_pH^0_{\fppf}(S^\prime, R^2f_*\GG_m)^{G_k} \lra 0.
$$
This completes the proof since  $V_p\Br(X_{k^s})^{G_k}\cong V_pH^0_{\fppf}(S^\prime, R^2f_*\GG_m)^{G_k}$ by arguments above and $\NS(X)\otimes_\ZZ\QQ_p\cong(\NS(X_{\bar{k}})\otimes_\ZZ\QQ_p)^{G_k}$ by Lemma \ref{nscompu} below.
\end{proof}
\begin{lem}\label{3vanishing}
$H^1_{\fppf}(S^\prime, R^1f_{*}\GG_m)$ is of finite exponent and $H^2_{\fppf}(S^\prime, R^1f_{*}\GG_m)=0$. 
\end{lem}
\begin{proof}
 $R^1f_{*}\GG_m$ is represented by the Picard scheme $P:= \Pic_{X/k}$. Write $P^0$ for $\Pic^0_{X/k}$ and $P^0_{\red}$ for $\Pic^0_{X/k,\red}$. By \cite{FGA}, there are exact sequences $$0\lra P^0\lra P \lra \NS_X\lra 0,$$ 
 and 
 $$0\lra P^0_{\red}\lra P^0\lra T \lra0,$$
 where $\NS_X$ is an \'etale group scheme over $k$ ( $\NS_X(k^s)=\NS(X_{\bar{k}})$ ) and $T$ is a finite group scheme over $k$. Since $P^0_{\red}$ and $\NS_X$ are smooth groups over $k$,
 $$H^i_{\fppf}(S^\prime, P^0_{\red})=0  \text{ and } H^i_{\fppf}(S^\prime, \NS_X)=0$$
 for $i\geq 1$. By \cite[Prop. 5.1]{Ma-Ro}, there is an exact sequence
 $$0\lra T\lra G_0\lra G_1\lra 0,$$
 where $G_0$ and $G_1$ are smooth groups of finite type over $k$. Thus,
 $$H^i_{\fppf}(S^\prime, T)=0 \text{ for $i\geq 2$}.$$
 The claims follow from the long exact sequences associated to the two exact sequences at the beginning of the proof.
\end{proof}

\begin{cor}\label{septoalgclosed}
The the kernel of the natural map $\Br(X_{k^s}) \lra \Br(X_{\bar{k}})$ is isomorphic to $H^1_{\fppf}(S^\prime, R^1f_{*}\GG_m)$, which is of finite exponent and vanishes if $\Pic_{X/k}$ is smooth over $k$.
\end{cor}
\begin{proof}
 $$0\lra H^1_{\fppf}(S^\prime, R^1f_{*}\GG_m)\lra H^2_{\fppf}(X_{k^s},\GG_m)\lra H^0_{\fppf}(S^\prime, R^2f_{*}\GG_m)$$
 is exact. By the lemma above, $H^1_{\fppf}(S^\prime, R^1f_{*}\GG_m)$ is of finite exponent. It suffices to show that 
 $$H^0_{\fppf}(S^\prime, R^2f_{*}\GG_m)\lra H^0_{\fppf}(\Spec(\bar{k}), R^2f_{*}\GG_m)=\Br(X_{\bar{k}})$$
 is injective. This follows from \cite[Lem. 3.3]{DAd}.
\end{proof}
\begin{lem}\label{nscompu}
There is a natural map which is al. isomorphism uniformly in $n$:
$$H^0_{\fppf}(S,R^1f_*\GG_m/p^n) \lra (\NS(X_{\bar{k}})/p^n)^{G_k}.$$
Moreover, there is a canonical isomorphism
$$\NS(X)\otimes_\ZZ\QQ_p\cong(\NS(X_{\bar{k}})\otimes_\ZZ\QQ_p)^{G_k}.$$
\end{lem}
\begin{proof}
    Notations as in the proof of Lemma \ref{3vanishing},
    $$R^1f_{*}\GG_m/p^n=P/p^n.$$
There are exact sequences of $\fppf$ sheaves
$$P^0/p^n\lra P/p^n\lra \NS_X/p^n\lra 0,$$
and
$$P^0_{\red}/p^n\lra P^0/p^n\lra T/p^n\lra 0.$$
$P^0_{\red}/p^n=0$ since $P^0_{\red}$ is an abelian variety. $P^0/p^n\cong T/p^n$ is killed by the order of $T$. Thus,
$$H^0_{\fppf}(S,P/p^n)\lra H^0_{\fppf}(S,\NS_X/p^n)$$ 
is an al. isomorphism uniformly in $n$. Since $\NS_X/p^n$ is an \'etale group scheme over $k$,
$$H^0_{\fppf}(S,\NS_X/p^n)= H^0_{\et}(S^\prime,\NS_X/p^n)^{G_k}=(\NS_X(k^s)/p^n)^{G_k}=(\NS(X_{\bar{k}})/p^n)^{G_k}.$$
For the second claim, as in the proof of Lemma \ref{3vanishing}, taking cohomology for the following exact sequences
 $$0\lra P^0\lra P \lra \NS_X\lra 0,$$ 
 and 
 $$0\lra P^0_{\red}\lra P^0\lra T \lra0,$$
 we get an exact sequence
 $$0\lra P^0(k^s)\lra P(k^s) \lra \NS_X(k^s)\lra H^1_{\fppf}(S^\prime, P^0).$$ 
 The last groups is of finite exponent. This shows that $\NS(X_{k^s})\lra \NS(X_{\bar{k}})$ is an al. isomorphism. Thus, $\NS(X_{k^s})\otimes_{\ZZ}\QQ_p\cong\NS(X_{\bar{k}})\otimes_{\ZZ}\QQ_p$. By \cite[\S, 2.2, p10]{Yua2}, $$\NS(X)\otimes_\ZZ\QQ_p\cong(\NS(X_{k^s})\otimes_{\ZZ}\QQ_p)^{G_k}.$$
 This completes the proof.
 
\end{proof}

\begin{lem}\label{bigtrick}
The natural map $\Br(X)[p^n] \lra \Br(X_{k^s})^{G_k}[p^n]$ is al. surjective uniformly in $n$.
\end{lem}
\begin{proof}
Firstly, assume that $X(k)$ is not empty. By \cite{CTS1} (cf. also \cite{Yua2}), the natural map $\Br(X)\lra \Br(X_{k^s})^{G_k}$ is al. surjective and there is a canonical exact sequence
\begin{equation}\label{formulaexactBr}
0 \lra \Br(k) \lra \Ker(\Br(X)\lra \Br(X_{k^s})^{G_k})\lra H^1(k,\Pic_{X/k}) \lra 0
\end{equation}
Write $M_X=:\Im(\Br(X)\lra \Br(X_{k^s})^{G_k})$ and $N_X:=\Ker(\Br(X)\lra \Br(X_{k^s})^{G_k})$. By the snake lemma, there is an exact sequence
$$
0\lra N_X[p^n] \lra \Br(X)[p^n] \lra M_X[p^n]\lra N_X/p^n.
$$
Let $Y \subseteq X$ be a smooth projective geometrically integral curve over $k$ obtained by taking hyperplane sections repeatedly (cf. \cite{Yua2}). By extending $k$, we assume that $Y(k)$ is not empty. Fix a $k$-point $P:\Spec(k) \lra Y$, then the first exact sequence is split. So $N_X=\Br(k)\oplus H^1(k,\Pic_{X/k})$. Similarly, there are two exact sequences above for $Y$ and $N_Y=\Br(k)\oplus H^1(k,\Pic_{Y/k})$. By functoriality and the fact $\Br(Y_{k^s})=0$ (so $M_Y=0$), it suffices to prove that the natural map
$$N_X/p^n \lra N_Y/p^n$$
is al. injective uniformly in $n$.  Thus, it suffices to show that 
$H^1(k, \Pic_{X/k})/p^n \lra H^1(k, \Pic_{Y/k})/p^n$
is al. injective uniformly in $n$. The natural map $H^1(k, \Pic^0_{X/k})\lra H^1(k, \Pic_{X/k})$ is al. isomorphism since $\NS(X_{k^s})$ is finitely generated. Hence, it suffices to show that
$H^1(k, \Pic^0_{X/k})/p^n \lra H^1(k, \Pic^0_{Y/k})/p^n$
is al. injective uniformly in $n$. This follows from the fact that there is an abelian group $D$ and a morphism $D \lra H^1(k, \Pic^0_{Y/k})$ such that the induced map 
$H^1(k, \Pic^0_{X/k})\oplus D \lra H^1(k, \Pic^0_{Y/k})$
is an al. isomorphism (cf. \cite{Yua2} for details). This proves the claim for $X_l/l$ where $l$ is some finite Galois extension of $k$. Let $G$ denote the Galois group $\Gal(l/k)$. Then claim for $X/k$ follows from the fact that
$$\Br(X)\lra \Br(X_l)^{G}$$
is an al. isomorphism (cf. \cite{Yua2} for details).
\end{proof}
\begin{lem}\label{cofiniteness}
The Brauer group $\Br(X_{\bar{k}})$ is an extension of a group of finite exponent by a divisible group of cofinite type. The Brauer group $\Br(X_{k^s})$ is an extension of a group of cofinite type by a group of finite exponent.
\end{lem}
\begin{proof}
There is an exact sequence 
$$0\lra \NS(X_{\bar{k}})\otimes_{\ZZ} \QQ_p/\ZZ_p\lra H^2_{\fppf}(X_{\bar{k}},\mu_{p^{\infty}})\lra \Br(X_{\bar{k}})[p^{\infty}]\to0.$$
By the proof of \cite[Prop. 5.9]{Ill}, $H^2_{\fppf}(X_{\bar{k}},\mu_{p^{\infty}})$ is an extension of group of finite exponent by a group of cofinite type. Hence, the claim is true for $\Br(X_{\bar{k}})[p^{\infty}]$. Since $\Br(X_{\bar{k}})(\non p)$ is of cofinite type, this proves the first claim. The second claim follows from Corollary \ref{septoalgclosed} and the first claim.
\end{proof}

\section{Proofs of Theorem \ref{mainthm} and its corollaries}
In this section, we will prove Theorem \ref{mainthm} and all corollaries in the introduction under the assumption of Theorem \ref{thekeythm} below which will be proved in the next section.
\begin{lem}\label{rigsha}
Let $f:\CX \lra U$ be a smooth projective morphism between smooth integral varieties over $\FF_p$ with a generic fiber $X/k$. Assume that $X$ is geometrically integral over $k$ and $X/k$ has at least one $k$-rational point. Then
$$\dim_{\QQ_p}(H^1_{\rig}(U, R^1f_*\CO_{\CX}^\dag)^{F=p})=\dim_{\QQ_p}(V_p\Sha(\Pic^0_{X/k,\red}))+\dim_{\QQ_p}(\Pic^0(X)\otimes_\ZZ \QQ_p)$$
\end{lem}
\begin{proof}
By the assumption, $X/k$ admits a section. This gives an Albanese morphism $a:X \lra A$ by \cite[no 236, Thm. 3.3 (iii)]{FGA}. By shrinking $U$ ( $R^1f_*\CO_{\CX}^\dag$ is of pure weight 1, the left side does not change for the reason of weights \cite{kedlaya2006fourier}), $A/k$ extends to an abelian scheme $g:\CA \lra U$. By the N\'eron model property of $g$, $a$ extends to a proper  morphism $a:\cX \lra \cA$.  The induced pullback for locally constant sheaves $R^1g_*\mu_{\ell^n} \lra R^1f_*\mu_{\ell^n}$ is an isomorphism. Thus for each $x\in U$, $a:\cX_{x} \lra \cA_{x}$ induces an isogeny between the reduced Picard varieties. This implies that the pull-back $a^*: R^1g_*\CO_{\CA}^\dag \lra R^1f_*\CO_{\CX}^\dag$ is isomorphism. To see this, by Kedlaya's full-faithfulness theorem \cite{kedlaya2004full}, it suffices to show that 
$$
a^*: R^1g_{\cris,*}\CO_{\CA/\ZZ_p}\otimes \QQ_p\lra R^1f_{\cris,*}\CO_{\CX/\ZZ_p}\otimes \QQ_p
$$
is an isomorphism of $F$-isocrystals. This follows from \cite[Prop. A.8 and Lem. A.9]{Mor} since the claim holds for fibers over a closed point in $U$ (cf. \cite{crys1motive}).

By Corollary \ref{appendspectral}, there is a Leray spectral sequence
$$E_2^{p,q}=H^p_{\rig}(U, R^qf_*\CO_{\CX}^\dag)\Rightarrow H^{p+q}_{\rig}(\CX).
$$
Following the pull-back trick in \cite[\S, 3.2]{Qin1}, these canonical maps
$$d^{1,1}_2:E_2^{1,1}\lra E_2^{3,0}$$
$$d_3^{0,2}:E_3^{0,2}\lra E_3^{3,0}$$
$$d_2^{0,2}:E_2^{0,2}\lra E_2^{2,1}$$
vanish. Thus, there are exact sequences
$$
0\lra K \lra H^{2}_{\rig}(\CX) \lra H^0_{\rig}(U, R^2f_*\CO_{\CX}^\dag) \lra 0
$$
and
$$0\lra H^2_{\rig}(U) \lra K \lra H^1_{\rig}(U, R^1f_*\CO_{\CX}^\dag)\lra 0.$$
Taking $F=p$ and using the pull-back trick in \cite[\S, 3.2]{Qin1} again, we get exact sequences
$$
0\lra K^{F=p} \lra H^{2}_{\rig}(\CX)^{F=p}\lra H^0_{\rig}(U, R^2f_*\CO_{\CX}^\dag)^{F=p} \lra 0
$$
and 
$$0\lra H^2_{\rig}(U)^{F=p} \lra K^{F=p} \lra H^1_{\rig}(U, R^1f_*\CO_{\CX}^\dag)^{F=p}\lra 0.$$
Since $H^1_{\rig}(U, R^1f_*\CO_{\CX}^\dag)^{F=p} \cong H^1_{\rig}(U, R^1g_*\CO_{\CA}^\dag)^{F=p}$ and $\Pic^0_{X/k,\red}\cong \Pic^0_{A/k}$, it suffices to prove the claim for the case that $\CX$ is an abelian scheme. Since the Tate conjecture for divisor is known for abelian varieties, $\dim_{\QQ_p}(H^0_{\rig}(U, R^2f_*\CO_{\CX}^\dag)^{F=p})=\dim_{\QQ_p}(\NS(X)\otimes_{\ZZ}\QQ_p)$ by Proposition \ref{propositionindependence}.
By \cite[Prop. 2.22]{EKQ}, 
$$
\dim_{\QQ_p}(H^2_{\rig}(\CX)^{F=p})=\dim_{\QQ_p}(\Pic(\CX)\otimes_{\ZZ}\QQ_p) + \dim_{\QQ_p}(V_p\Br_{\nr}(K(X))),
$$ 
$$
\dim_{\QQ_p}(H^2_{\rig}(U)^{F=p})=\dim_{\QQ_p}(\Pic(U)\otimes_{\ZZ}\QQ_p) + \dim_{\QQ_p}(V_p\Br_{\nr}(k)).
$$ 
Hence
\begin{align*}
&\dim_{\QQ_p}(H^1_{\rig}(U, R^1f_*\CO_{\CX}^\dag)^{F=p}) 
- \dim_{\QQ_p}(V_p\Sha(\Pic^0_{X/k})) \\
&= \dim_{\QQ_p}(\Pic(\CX) \otimes_{\ZZ} \QQ_p) 
+ \dim_{\QQ_p}(V_p \Br_{\nr}(K(X))) \\
&\quad - \dim_{\QQ_p}(V_p\Sha(\Pic^0_{X/k})) 
- \dim_{\QQ_p}(\NS(X) \otimes_{\ZZ} \QQ_p) \\
&\quad - \dim_{\QQ_p}(\Pic(U) \otimes_{\ZZ} \QQ_p) 
- \dim_{\QQ_p}(V_p\Br_{\nr}(k)) \\
&= \dim_{\QQ_p}(\Pic(\CX) \otimes_{\ZZ} \QQ_p) 
- \dim_{\QQ_p}(\Pic(U) \otimes_{\ZZ} \QQ_p) 
- \dim_{\QQ_p}(\NS(X) \otimes_{\ZZ} \QQ_p) \\
&= \dim_{\QQ_p}(\Pic(X) \otimes_{\ZZ} \QQ_p) 
- \dim_{\QQ_p}(\NS(X) \otimes_{\ZZ} \QQ_p) \\
&= \dim_{\QQ_p}(\Pic^0(X) \otimes_{\ZZ} \QQ_p).
\end{align*}

The second equality follows from the exact sequence
$$
0\longrightarrow V_p\Br_{\nr}(k)\longrightarrow V_p\Br_{\nr}(K(X)) \longrightarrow V_p\Sha(\Pic^0_{X/k})\longrightarrow 0,
$$
which was proved in \cite[Thm. (D)]{EKQ}.
The third equality follows from the exact sequence ( cf. \cite[Prop. 1.3]{EKQ})
$$
0\lra \Pic(U) \lra \Pic(\CX) \lra \Pic(X)\lra 0.
$$
\end{proof}
\subsubsection*{Proof of Theorem \ref{mainthm}}
\begin{proof}
By Proposition \ref{propositionindependence}, the Tate conjecture for divisors on $X$ is equivalent to 
$$\dim_{\QQ_p}
(\NS(X)\otimes_\ZZ\QQ_p)=\dim_{\QQ_p}(H^0_{\rig}(U, R^2f_*\CO_{\CX}^\dag)^{F=p}
).$$
By Theorem \ref{thekeythm}, this is equivalent to
$$
\dim_{\QQ_p}
(\NS(X)\otimes_\ZZ\QQ_p)=\dim_{\QQ_p}(\varprojlim_{n}H^0_{\fppf}(\Spec(k), R^2f_{*}\mu_{p^n})\otimes_{\ZZ_p}\QQ_p),
$$
which is equivalent to $V_p\Br(X_{k^s})^{G_k}=0$ by Proposition \ref{firstexact}. By Lemma \ref{cofiniteness}, $\Br(X_{k^s})^{G_k}[p^\infty]$ is al. isomorphic to a divisible group of cofinite type. Hence, $V_p\Br(X_{k^s})^{G_k}=0$ is equivalent to
that $\Br(X_{k^s})^{G_k}[p^\infty]$ is of finite exponent. By \cite[Prop. 6.6]{Pal} and \cite{CHT}, this is also equivalent to $\Br(X_{k^s})^{G_k}$ is of finite exponent. This proves the first claim.

For the second claim, we will first prove it under the assumption $X(k)\neq \varnothing$. We will remove this assumption later. To prove the claim, it suffices to show that
$$V_p\Br_{\nr}(K(X)) \lra V_p\Br(X_{k^s})^{G_k}$$
is surjective since $\Br(X_{k^s})^{G_k}$ is an extension of a group of finite exponent by a group of cofinite type (cf. Lemma \ref{cofiniteness}). Thus, $V_p\Br(X_{k^s})^{G_k}$ is a finite dimensional $\QQ_p$-linear space. Thus, it suffices to show 
$$\dim_{\QQ_p}(V_p\Br_{\nr}(K(X)))=\dim_{\QQ_p}(\Ker(V_p\Br_{\nr}(K(X)) \lra V_p\Br(X_{k^s})^{G_k}) +\dim_{\QQ_p}(V_p\Br(X_{k^s})^{G_k}).$$
By the first part of Corollary \ref{Artin-Grothendieck} (cf. \cite{EKQ} for the proof), it suffices
to show 
$$\dim_{\QQ_p}(V_p\Br_{\nr}(K(X)))=\dim_{\QQ_p}(V_p\Br_{\nr}(k)) + \dim_{\QQ_p}(V_p\Sha(\Pic^0_{X/k,\red})) +\dim_{\QQ_p}(V_p\Br(X_{k^s})^{G_k}).$$
By spreading out, $X/k$ can be extended to a morphism $f:\CX \lra U$ as in Lemma \ref{rigsha}.
By the proof of Lemma \ref{rigsha},
\begin{align*}
\dim_{\QQ_p}(H^{2}_{\rig}(\CX)^{F=p})=\dim_{\QQ_p}(H^0_{\rig}(U, R^2f_*\CO_{\CX}^\dag)^{F=p})+ \dim_{\QQ_p}(H^2_{\rig}(U)^{F=p}) \\
+ \dim_{\QQ_p}( H^1_{\rig}(U, R^1f_*\CO_{\CX}^\dag)^{F=p})
\end{align*}
Plug the following equations into the above equation, we get the desired result.
$$
\dim_{\QQ_p}(H^2_{\rig}(\CX)^{F=p})=\dim_{\QQ_p}(\Pic(\CX)\otimes_{\ZZ}\QQ_p) + \dim_{\QQ_p}(V_p\Br_{\nr}(K(X)))
$$ 
$$
\dim_{\QQ_p}(H^2_{\rig}(U)^{F=p})=\dim_{\QQ_p}(\Pic(U)\otimes_{\ZZ}\QQ_p) + \dim_{\QQ_p}(V_p\Br_{\nr}(k))
$$ 
$$
\dim_{\QQ_p}(H^0_{\rig}(U, R^2f_*\CO_{\CX}^\dag)^{F=p})=\dim_{\QQ_p}
(\NS(X)\otimes_\ZZ\QQ_p)+\dim_{\QQ_p}(V_p\Br(X_{k^s})^{G_k})
$$
$$
\dim_{\QQ_p}(H^1_{\rig}(U, R^1f_*\CO_{\CX}^\dag)^{F=p})=\dim_{\QQ_p}(\Pic^0(X)\otimes_\ZZ \QQ_p)+\dim_{\QQ_p}(V_p\Sha(\Pic^0_{X/k,\red}))
$$
This proves the claim for the case $X(k)\neq \varnothing$. In general, there exists a finite Galois extension $l/k$ such that $X(l) \neq \varnothing$. So $\Br_{\nr}(K(X_l)) \lra \Br(X_{k^s})^{G_l}$ is al. surjective. Thus, the claim for $X/k$ follows from the fact that
$\Br_{\nr}(K(X)) \lra \Br_{\nr}(K(X_l))^{\Gal(l/k)}$
is an al. isomorphism.
\end{proof}
\subsubsection*{Proof of Corollary \ref{cor1.3}}
\begin{proof}
The first claim follows directly from Theorem \ref{mainthm}. By the assumption and Corollary \ref{septoalgclosed}, the natural map $\Br(X_{k^s})\lra \Br(X_{\bar{k}})$ is injective, it follows from the proof of \cite[Thm. 5.2]{DAd} that $\Br(X_{\bar{k}})$ is of cofinite type. Thus, $\Br(X_{k^s})^{G_k}$ is of cofinite type.
\end{proof}

\subsubsection*{Proof of Corollary \ref{fgfield}}
\begin{proof}
By purity of Brauer groups \cite{Ces},
$\Br_{\nr}(K(X))=\Br(\CX)$. By Theorem \ref{mainthm},
$$\Br(\CX)\lra \Br(X_{K^s})^{G_K}$$
is al. surjective. Since it factors through 
$$\Br(\CX_{k^s})^{G_k} \lra \Br(X_{K^s})^{G_K},
$$
the claim follows.
\end{proof}

\section{Syntomic Cohomology and Crystalline Cohomology}
Let $X$ be a smooth projective geometrically integral variety over a field $k$ finitely generated over $\FF_p$. By spreading out $X/k$, we get a smooth projective morphism $f:\CX \lra U$ between smooth integral varieties over $\FF_p$. Recall P\'al's formulation of the Tate conjecture using rigid cohomology \cite[\S\, 6]{Pal}: the natural injective map
$$
\NS(X)\otimes_\ZZ\QQ_p\lra H^0_{\rig}(U, R^2f_*\CO_{\CX}^\dag)^{F=p}
$$
is surjective. P\'al \cite[Prop. 6.6]{Pal} showed that it is equivalent to the $\ell$-adic Tate conjecture. By Proposition \ref{propconvergentformulation}, it is also equivalent to the surjectivity of the natural injective map
$$\NS(X)\otimes_\ZZ\QQ_p\lra H^0_{\conv}(U, R^2f_*\cO_{\CX/K})^{F=p}.$$
The goal of this section is to prove the following theorem.
\begin{thm}\label{thekeythm}
There is a $\QQ_p$-linear isomorphism
$$
 \varprojlim_{n}H^0_{\fppf}(\Spec(k), R^2f_{*}\mu_{p^n})\otimes_{\ZZ_p}\QQ_p \cong H^0_{\conv}(U, R^2f_*\CO_{\CX/K})^{F=p}.
$$
\end{thm}

\begin{rem}\label{remarkkeythm}
    By Proposition \ref{ertlberthelot}, we have an isomorphism
    $$H^0_{\rig}(U, R^2f_*\CO^\dagger_{\CX/K})^{F=p}\cong H^0_{\conv}(U, R^2f_*\CO_{\CX/K})^{F=p}$$
\end{rem}

This is divided into two steps, for the first we show

\begin{prop}\label{theoremtowardscrystalline}
There is an al. isomorphism uniformly in $n$ 
$$H_{\syn}^0(S,R^2f_{\syn*}\mu_{p^n,X})\stackrel{\cong}\longrightarrow H^0_{\crys}(S,R^2f_{\crys*}\cO_{X/W_n})^{F=p}.$$
\end{prop}

Then we relate crystalline cohomology of the generic fiber to rigid cohomology of $\cX$

\begin{prop}
\label{crystallinegenerictorigid}
    There is a $\bQ_p$-linear isomorphism
    $$H^0_{\crys}(S,R^2f_{\crys*}\cO_{X/W})^{F=p}[1/p]\cong H^0_{\conv}(U,R^2f_*\cO_{\cX/K})^{F=p}.$$
\end{prop}

\subsection{Flat cohomology}
 Let $v: \syn(X) \lra \fppf(X)$ be the obvious continuous functor between sites
\begin{lem}\label{fppfsyn}
 Let $G$ be a smooth commutative group scheme over $X$ and $v: \syn(X) \lra \fppf(X)$ be the morphism of sites. Then $R^qv_*G=0$ for any $q>0$.
\end{lem}
\begin{proof}
This follows from the the proof of \cite[Chap. III, 3.9]{Mil3}. In fact, it is the sheafification on $\syn(X)$ of the higher direct image from the $\fppf$ site to the category of syntomic $X$-scheme equipped with the \'etale topology. For $q>0$, the second is 0. So it is also $0$ for $q>0$.
\end{proof}
$f$ induces a commutative diagram of continuous functors between sites
\begin{displaymath}
\xymatrix{
\syn(S)\ar[r]^v \ar[d]^f & \fppf(S) \ar[d]^f \\
\syn(X) \ar[r]^v &  \fppf(X)
}
\end{displaymath}
Consider the Kummer exact sequence $ 0\lra \mu_p^n \lra \GG_m\stackrel{p^n}{\lra}\GG_m\lra 0$ on $\fppf(X)$. It is still exact after taking $v_*$. Thus, by the Lemma above, $R^qv_*\mu_{p^n}=0$ (also true on $S$) for any $q>0$.
Thus, there is a Leray spectral sequence
$$
E^{p,q}_2=R^pv_*R^qf_*\mu_{p^n,X} \Rightarrow R^{p+q}f_*(v_*\mu_{p^n,X})
$$
\begin{lem}\label{fppfsyncoho}
 The edge map 
 $R^{q}f_*(v_*\mu_{p^n,X})\lra v_*R^qf_*\mu_{p^n,X}$ is an isomorphism for $q\leq 2$.
\end{lem}
\begin{proof}
$f_*\mu_{p^n,X}=\mu_{p^n,S}$ Since $f$ is geometrically connected. This shows that $E_2^{p,0}=0$ for $p>0$. $T:=R^1f_*\mu_{p^n,X}=\Pic_{X/k}[p^n]$ is a finite flat group scheme over $k$. By \cite[Prop. 5.1]{Ma-Ro}, there is an exact sequence
 $$0\lra T\lra G_0\lra G_1\lra 0,$$
 where $G_0$ and $G_1$ are smooth group of finite type over $k$. $G_0\lra G_1$ is faithful flat. So it is exact on $\syn(S)$ since $G_0\lra G_1$ is in fact a syntomic cover. By the Lemma above, $R^pv_*T=0$ for $p>0$. This proves the claim.
\end{proof}
\begin{cor}
There is a canonical isomorphism
$$
H^0_{\fppf}(S, R^2f_*\mu_{p^n,X})\cong H^0_{\syn}(S, R^2f_*v_*\mu_{p^n,X}).
$$
\end{cor}
In the following, we will first prove Theorem \ref{thekeythm} for curves by a direct computation. Then, we prove it for abelian varieties by reducing to curves. Then, we prove the theorem for general $X$ by reducing to curves and abelian varieties.
\begin{lem}\label{lemmainjofflat}
Let $f: X\lra S=\Spec(k)$ be smooth projective geometrically integral curve. Then
$$H^0_{\fppf}(S,R^2f_*\mu_{p^n})\cong H^0_{\fppf}(\Spec(\bar{k}),R^2f_*\mu_{p^n})\cong \ZZ/p^n.$$
\end{lem}
\begin{proof}
By the proof of Proposition \ref{firstexact}, there is an exact sequence
$$ 0\lra H^0_{\fppf}(S,R^1f_*\GG_m/p^n) \lra H^0_{\fppf}(S, R^2f_*\mu_p^n) \lra  H^0_{\fppf}(S, R^2f_*\GG_m).$$
 By \cite[Lem. 3.3]{DAd}, the natural map
  $$H^0_{\fppf}(S, R^2f_*\GG_m)\lra H^0_{\fppf}(\Spec(\bar{k}), R^2f_*\GG_m)$$
is injective. By the fact $H^0_{\fppf}(\Spec(\bar{k}), R^2f_*\GG_m)=\Br(X_{\bar{k}})=0$, the second arrow in the exact sequence is an isomrphism.
Since $\Pic_{X/k}$ is smooth, the proof of Lemma \ref{nscompu} actullay implies that
 $$
 H^0_{\fppf}(S,R^1f_*\GG_m/p^n)\cong ((\NS(X_{\bar{k}})/p^n)^{G_k}\cong \ZZ/p^n
 $$
By \cite[Lem. 3.3]{DAd}, the natural map
$$
H^0_{\fppf}(S,R^2f_*\mu_{p^n})\lra H^0_{\fppf}(\Spec(\bar{k}),R^2f_*\mu_{p^n})
$$
is injective. So it is an isomorphism.
\end{proof}

\subsection{Crystalline cohomology}

\begin{defn}
	\begin{enumerate}
        \item Let 
        $(X/W_n)_{\CRYS,\SYN}\stackrel{v_X}\to X_{\SYN}$
        be the natural morphism of topos and $\cO^{\crys}_{X,n}:=v_*(\cO_{X/W_n})$. 
        \item Let $p:\syn(X)\to \SYN(X)$ be the natural continuous functor.
        Since $p_*=p^s$ is exact, we also use $\cO_{n}^{\crys}$ to denote $p_*\cO_n^{\crys}$ when there is no confusion.
		\item Define the sheaf $\cJ_n$ on $SYN(X)$ via the exact sequnce
		$$\begin{CD}
		0@>>> \cJ_n@>j_n>> \cO^{\crys}_n@>>> \bG_a@>>> 0.
        \end{CD} $$
	\end{enumerate}
\end{defn}

Define $\cI_n$ as the quotient $\cJ_{n+1}/\cJ_1$.  There is an exact sequence
$$\begin{CD}
		0@>>> p^n\cO^{\crys}_{n+1}/\cJ_1@>>> \cI_n @>i_n>>\cJ_n@>>> 0.
\end{CD}$$
Our proof depends on the following deep syntomic comparison result of Fontaine which was used in Bauer's proof:
\begin{prop}[{\cite[Thm. 3.4]{Bauer1992}\cite[Lem. 9.9]{trihan2018comparisontheoremsemiabelianschemes}}]\label{Bauerexact}
	There is an exact sequence of sheaves on $syn(X)$
	$$\begin{CD}
	0@>>> \mu_{p^n}@>>> \cI_n@>1-p^{-1}F>> \cO^{\crys}_n@>>> 0
		\end{CD}$$
\end{prop}

\begin{rem}
    By definition, the following diagram commutes
    $$\begin{CD}
        \cI_n@>1-p^{-1}F>> \cO_{n}^{\crys}\\
        @Vj_n\circ i_nVV @VVpV\\
        \cO_{n}^{\crys}@>p-F>> \cO_{n}^{\crys}
    \end{CD}$$
    such that kernel and cokernel of each vertical map is kill by $p^2$. Thus the fiber of each line are al. quasi-isomorphic uniformly in $n$.
\end{rem}

Apply $Rf_{\syn*}(-)$ to the complex in proposition \ref{Bauerexact} and write $g_i$ to be the map $1-p^{-1}F$ between $R^if_{\syn*}\cJ_n$ and $R^if_{syn*}\cO_n^{crys}$, we get an exact sequence
$$\begin{CD}
	0@>>>\coker(g_1)@>>> R^2f_{syn*}\mu_{p^n}@>>>\ker(g_2)@>>>0
\end{CD}
$$
which induces a long exact sequence 
$$0\to H^0(S,\coker(g_1))\to H^0(S,R^2f_{\syn*}\mu_{p^n})\to H^0(S,\ker(g_2))\stackrel{\gamma_X}\to H^1(S,\coker(g_1)).$$
We need to show that $H^0(S,\coker(g_1))$ is al. zero uniformly in $n$ and 
the last map $\gamma_X$ is al. zero uniformly in $n$.

\subsubsection{Change of topology}

\begin{lem}\label{lemmachangetopo2}
	Consider the following commutative diagram
	$$\begin{CD}
	X_{\CRYS,\SYN}@>f_{\CRYS,\SYN}>> S_{\CRYS,\SYN}\\
	@Vc_{\CRYS,SZ}VV @VVc_{\CRYS,SZ}V \\
	X_{\CRYS,\ZAR}@>f_{\CRYS,\ZAR}>> S_{\CRYS,\ZAR}.
\end{CD}$$
Let $\cE^i_{X,\SYN,n}:= R^if_{\CRYS,\SYN*}\cO_{X/W_n}$ and $\cE^i_{X,\ZAR,n}:= R^if_{\CRYS,\ZAR*}\cO_{X/W_n}$ similarly, then 
$$\cE^i_{X,\SYN,n}\cong c_{\CRYS,SZ}^*\cE^i_{X,\ZAR,n}.$$
\end{lem}

\begin{proof}
 We consider the presheaf $\cP$ on the category $\CRYS(S/W_n)$ 
	$$(U,T,\delta)\to H^i((X_U/T)_{\CRYS,\SYN},\cO_{X_U/T}).$$
	The cohomology group is the same as $H^i((X_U/T)_{\CRYS,\ZAR},\cO_{X_U/T})$ (see \cite[1.1.19]{berthelot2006theorie}). 
	By \cite[Prop. 1.1.11]{berthelot2006theorie}, the syntomic (resp. Zariski) sheaf associated to $\sP$ is $\cE^i_{X,\SYN,n}$ (resp. $\cE^i_{X,\ZAR,n}$) and the claim follows.	
\end{proof}

\subsubsection{Crystal property of higher direct image}
Recall quasi-smoothness of morphisms between schemes is defined in \cite[IV Definition 1.5.1]{berthelot2006cohomologie}. Note that all smooth maps are quasi-smooth (\cite[IV 1.5.5]{berthelot2006cohomologie}) and for a characteristic $p$ field $k$ and $C(k)$ is a Cohen ring of $k$. Then $C_n(k)=C(k)/p^nC(k)$ is quasi-smooth over $W_n$ because $k$ has a $p$-base (see \cite[Proposition 1.2.6]{berthelot2007theorie} and \cite[Remark A.4]{Mor}). 

\begin{lem}\label{lemmacomputepushforward1}
	For each $(U,T,\delta)\in \CRYS(S/W_n)_{\ZAR}$, Zariski locally on $T$ there exists a PD-morphism $v:T\to D_n:=\Spec(C_n(k))$ and 
	$$(Rf_{\CRYS*}\cO_{X/W_n})_T\cong Lv^*R\Gamma_{\crys}(X/D_n)$$
	where $R\Gamma_{\crys}(X/D_n)$ is regarded as a quasi-coherent Zariski sheaf over $D_n$. \end{lem} 

\begin{proof}
The sheaf $Rf_{\CRYS*}\cO_{X/W_n}$ is the sheafification  (in the derived $\infty$-category) of $(U',T',\delta')\to R\Gamma_{\crys}(X_{U'}/T')$ under the Zariski-topology. So 
	$$(Rf_{\CRYS*}\cO_{X/W_n})_T\cong Rf_{X_U/T*}\cO_{X_U/T}.$$

To see the existence of $v$, recall that the PD-envelope of $S\to D_n$ is $D_n$ ($p$ is the maximal ideal) so the result follows from $C_n$ being a quasi-smooth $W_n$-algebra.

	Assume there exists a PD-morphism $v:T\to C_n$, we apply smooth base change (\cite[Corollary 7.12]{berthelot2006theorie}) to the diagram 
	$$\xymatrix{ X_U\ar^{g'}[r]\ar_{f'}[d] & X\ar^f[d]\\
	U\ar^g[r] & S}$$
	and get 
	$$Lg_{\crys}^*Rf_{\crys*}\cO_{X/W_n}\simeq Rf'_{\crys*}Lg'^{*}_{\crys}\cO_{X/W_n}.$$
	The claim follows as 
 $$(Lg_{\crys}^*Rf_{crys*}\cO_{X/W_n})_T\simeq Lv^*(Rf_{\crys*}\cO_{X/W_n})_{D_n}\simeq Lv^*R\Gamma_{\crys}(X/D_n)$$ 
 and 
$$(Rf'_{\crys*}Lg'^*_{\crys}\cO_{X/W_n})_T\simeq Rf_{X_U/T*}\cO_{X_U/T}.$$
\end{proof}

Consider the following commutative diagrams of topos
$$\begin{CD}
	X_{\CRYS,\SYN}@>v_X>> X_{\SYN}\\
	@Vf_{\CRYS,\SYN}VV @Vf_{\SYN}VV \\
	S_{\CRYS,\SYN}@>v_S>> S_{\SYN}.
	\end{CD}$$
and continuous functors between sites
$$\begin{CD}
	\syn(S)@>p_S>> \SYN(S)\\
	@Vf_{\syn}VV @Vf_{\SYN}VV\\
	\syn(X)@>p_X>> \SYN(X).
\end{CD}$$
	
By \cite[Proposition 1.17]{Bauer1992}, $Rp_{X*}Rv_{X,*}\cO_{X/W_n}\cong p_{X*}v_{X,*}\cO_{X/W_n}=\cO^{\crys}_{X,n}$. So
\begin{equation}\label{formulafirstes}
	p_{S*}Rv_{S*}Rf_{\CRYS,\SYN*}\cO_{X/W_n}\simeq Rf_{\syn*}\cO^{\crys}_{X,n}.
\end{equation}

\begin{prop}\label{propositionalmostcrystal}
 The edge map 
$$R^if_{\SYN*}\cO^{\crys}_{X,n}\to v_{S*}\cE^i_{X,\SYN,n}$$
is an al. isomorphism uniformly in $n$.
\end{prop}

\begin{proof}
 Apply \cite[Prop. 1.17]{Bauer1992} and lemma \ref{lemmachangetopo2}, it suffices to check there exist crystals $\cE^i_{X/S,n}$ such that  $\cE^i_{X,\ZAR,n}$ is al. isomorphic to $\cE^i_{X/S,n}$ uniformly in $n$. For this, we follow Morrow's method: Let $C$ denote a Cohen ring of $k$, $C_n$ denote $C/p^n$, $D_n:=\Spec C_n$ and $D:=\Spec C$. The same argument as in the proof of \cite[Lem. A.2]{Mor} provides a morphism $\psi_n^i:\cE^i_{X/S,n}\to \cE^i_{X,\ZAR,n}$ where $\cE^i_{X/S,n}$ is the crystal constructed by Morrow in the proof of \cite[Lem. A.2]{Mor} whose value on affine $(U,T,\delta)$ is $g^*_TH^i(X/D_n)$ for any chosen map $g_T:T\to D_n$ lifting $U\to S$. Note that although the site considered there is the small site, \cite[Lem. A.3]{Mor} applies to big sites. Restricted to $T$, $\psi_{n,T}^i$ is given by the edge map
 $$\psi_{n,T}^i: g^*_TH^i(X/D_n)\to H^i(Lg^*_TR\Gamma_{\crys}(X/D_n))\simeq (R^if_{\CRYS*}\cO_{X/W_n})_T.$$
 We need to see the first map is an al. isomorphism uniform in $n$ and $g_T$.\par
  By \cite[Theorem 7.24]{BerthelotOgus1978}, the limit complex $R\Gamma_{\crys}(X/D)$ is a perfect complex with finite Tor-amplitude and $R\Gamma_{\crys}(X/D_n)\simeq R\Gamma_{\crys}(X/D)\otimes^L_{D} D_n$. Note that $D$ is a complete DVR thus the torsion part of $H^i_{crys}(X/D)$ is uniformly bounded and $H^i_{crys}(X/D)\otimes_D D_n$ is al. isomorphic to $H^i_{crys}(X/D_n)$ uniformly in $n$. So for each $i\geq 0$, the edge map 
$$g_T^*H^i_{\crys}(X/D_n)\to H^i(Lg_T^*R\Gamma_{\crys}(X/D_n))$$
is an al. isomorphism uniformly in $n$ and $g_T$.
 \end{proof}

 \begin{cor}
     There is an al. isomorphism uniformly in $n$
     $$H^0_{\syn}(S,R^if_{\syn*}\cO_{X,n}^{\crys})\longrightarrow H^0_{\crys}(S,R^if_{\crys*}\cO_{X/W_n}).$$
 \end{cor}

 \begin{proof}
     Pushforward the edge map in proposition \ref{propositionalmostcrystal} along $p_S$ and take cohomology, we get an al. isomorphism uniformly in $n$
     $$H^0_{\syn}(S,R^if_{\syn*}\cO_{X,n}^{\crys})\longrightarrow H^0(S_{\CRYS,\SYN},\cE^i_{X,\SYN,n}).$$
     We also see in the proof that $\cE^i_{X,\SYN,n}$ is al. isomorphic to $\cE^i_{X/S,n}$ (coherent crystals in Zariski or syntomic topology coincide) uniformly in $n$. As a coherent crystal 
     $$H^0(S_{\CRYS,\SYN},\cE^i_{X/S,n})\cong H^0(S_{\CRYS,\ZAR},\cE^i_{X/S,n})\cong H^0_{\crys}(S,\cE^i_{X/S,n}).$$
     Pushforward the al. isomorphism $\cE^i_{X/S}\to \cE^i_{X,\ZAR,n}$ to small crystalline site, we get an al. isomorphism $H^0_{\crys}(S,\cE^i_{X/S})\to H^0_{\crys}(S,R^if_{\crys*}\cO_{X/W_n})$ uniformly in $n$.
 \end{proof}

\begin{lem}\label{takinglimit}
Let $i_n:(S/W_n)_{\crys} \longrightarrow (S/W)_{\crys}$ be the canonical morphism of topos and $\cE^i_{X/S}$ be the crystal constructed by Morrow in the proof of \cite[Lem. A.2]{Mor} (recalled in proof of proposition \ref{propositionalmostcrystal}). Then
$$
i_n^*\cE^2_{X/S} \lra R^2f_{\crys*}\cO_{X/W_{n}}
$$
is an al. isomorphism uniformly in $n$.
\end{lem}
\begin{proof}
 The claim follows from Morrow's proof of \cite[Lemma A.2]{Mor} and that 
$$H^i_{\cris}(X/D)/p^n \lra H^i_{\cris}(X/D_n)$$
is an al. isomorphism uniformly in $n$ for any $i\geq 0$ (cf. \cite[Cor. 7.25]{BerthelotOgus1978}).
\end{proof}

\subsection{Relation with Albanese}
Let $X$ be a smooth projective geometrically connected variety over  field $k$, there is an Albanese variety $A$ associated to $X$ which is the dual abelian variety of the reduced part of the Picard scheme $\Pic^0_{X/k,\red}$ (cf. \cite[no 236, Thm. 3.3 (iii)]{FGA}) . The motivation is the following result \cite{crys1motive}
\begin{thm}\label{theoremcrystallinemotive}
	Let $X$ be a smooth and projective over a perfect field $k$ of characteristic $p > 0$. Let $T_{\crys}(-)$ denote the covariant Dieudonne module. Then, there is a canonical isomorphism
	$$T_{\crys}(\Pic^0_{X/k,\red})\cong H^1_{crys}(X/W(k)).$$
\end{thm}

Assume there is a rational point $a\in X$. Let $a:X\to A$ be the Albanese morphism induced by $a$. By definition, $a$ induced an isomorphism between the reduced Picard varieties \cite[no 236, Thm. 3.3 (iii)]{FGA}. So we have
$$\xymatrix{ X\ar^{a}[rr] \ar_{f}[rd] & & A\ar^{g}[ld]\\
& S &
}$$
and $a^*$ induces a map $\cE^1_{A,\SYN,n}\to \cE^1_{X,\SYN,n}$ which is the sheafification of the corresponding map $a^*:\cE^1_{A,\ZAR,n}\to \cE^1_{X,\ZAR,n}$ (Lemma \ref{lemmachangetopo2}). By crystalline Dieudonne theory, we know $\cE^1_{A,\ZAR,n}$ is a coherent crystal (cf \cite[Sec. 2.5]{berthelot2006theorie}). 

\begin{prop}\label{propositionrelateAlbanese}
 The morphism $a^*:\cE^1_{A,\ZAR,n}\to \cE^1_{X,\ZAR,n}$ is an al. isomorphism uniformly in $n$.
\end{prop}

\begin{proof}
Choose a Cohen ring $C=C(k)$, let $C_n:=C_n(k):=C(k)/p^n$ and $D_n:=\Spec(C_n)$. It suffices to check the statement locally in $\CRYS(S)_{\ZAR}$. Let $(U,T,\delta)\in \CRYS(S)_{\ZAR}$ and apply lemma \ref{lemmacomputepushforward1} we may assume there exists a PD-morphism $v:T\to D_n$.
 Via the proof of proposition \ref{propositionalmostcrystal}, for each $p\geq 0$, the edge map 
$$v^*H^p_{\crys}(X/D_n)\to H^p(Lv^*R\Gamma_{\crys}(X/D_n))$$
is an al. isomorphism uniformly in $n$ and $v$. Combine it with Lemma \ref{lemmacomputepushforward1}, we get an al. isomorphism uniformly bounded in $n$ and $v$
$$v^*H^p_{\crys}(X/D_n)\longrightarrow \cE^p_{X,\ZAR,n}|_T.$$
 Similarly, it works for the Albanese variety $A$. Take $p=1$, it is enough to see 
 $$H^1_{\crys}(X/D)\to H^1_{\crys}(A/D)$$
 induced by $a$ is an al. isomorphism uniformly bounded in $n$. To show the statement it is enough to show it after base change to $\overline{W}=W(\overline{k})$ which is faithfully flat $C$-algebra. Let $\overline{D}:=\Spec\overline{W}$, via smooth base change \cite[Coro 7.12]{BerthelotOgus1978} we get 
 $$H^1_{\crys}(X/D)\otimes_{D} \overline{W}\cong H^1_{\crys}(X_{\overline{k}}/\overline{D}).$$
 Now the result follows from Theorem \ref{theoremcrystallinemotive} and the duality between Albanese and Picard scheme. 
 \end{proof}

\section{Proof of Proposition \ref{theoremtowardscrystalline}}
 \subsection{Case of a curve}
 Assume $\dim X=1$.
 
 Base change to $\overline{S}=\Spec \overline{k}$ we have the following diagram 
 $$\begin{CD}
 	0@>>> H^0(S,\coker(g_1))@>>> H^0(S,R^2f_{syn*}\mu_{p^n})@>\beta_4>>H^0(S,\ker(g_2))@>\gamma_X>> H^1(S,\coker(g_1))\\
 	@. @V\beta_1VV @V\beta_2VV @V\beta_5VV @VVV\\
 	0@>>> H^0(\overline{S},\coker(g_1)_{\overline{S}})@>>> H^0(\overline{S},R^2f_{syn*}\mu_{p^n,X_{\overline{k}}})@>\beta_3>>H^0(\overline{S},\ker(g_2)_{\overline{S}})@>\overline{\gamma}_X>> H^1(\overline{S},\coker(g_1)_{\overline{S}}
 \end{CD}$$
 In the following, we show that $\beta_5$ is al. injective uniformly in $n$ (cf. Corollary \ref{beta5injective}) and $\beta_3$ is an al. isomorphism uniformly in $n$ (cf. Lemma \ref{beta3iso}). By Lemma \ref{lemmainjofflat}, $\beta_2$ is already an isomorphism. This shows that $\beta_4$ is an al. isomorphism uniformly in $n$. The digram chasing shows that
 $H^0(S,\coker(g_1))$ and the image of $\gamma_X$ can be killed by a positive integer independent of $n$.

  \begin{lem}\label{lemmaflatnessofhigherimage}
     We have $R^if_{\CRYS*}\cO_{X/W_n}$ ($i=0,1,2$) is a coherent crystal over $\SYN(S)$.
 \end{lem}
 \begin{proof}
      Through the proof of proposition \ref{propositionalmostcrystal}, the obstruction preventing $R^if_{\CRYS*}\cO_{X/W_n}$ from being a crystal lies in torsions in cohomology groups $H^i_{\crys}(X/D_n)$. So it suffices to see $H^i_{\crys}(X/D_n)$ ($i=0,1,2$) are flat $D_n$-modules. Since $\overline{W}_n$ is a faithfully flat $C_n$-algebra, apply smooth base change, it suffices to see it when $k$ is algebraically closed. Now the claim for $H^0$ is clear and Poincare duality gives $H^2$. For $H^1$, one can apply Theorem \ref{theoremcrystallinemotive}.
 \end{proof}
 
\begin{rem}
Any smooth projective geometrically connnected curve over a separable closed field $k$ of characteristic $p>0$ admits a lifting to $C(k)$.
This follows from theory of moduli of curves.
\end{rem}

So the edge map $R^2f_{\syn*}\cO^{\crys}_n\to v_*R^2f_{\crys*}\cO_{X/W_n}$ is an isomorphism.

 \begin{lem}\label{keyinjective}
 The natural map
$$
H^0(S,R^2f_{\crys*}\cO_{X/W_n})\rightarrow H^0(\overline{S},R^2f_{\crys*}\cO_{\overline{X}/W_n})
$$
is injective.
 \end{lem}
 \begin{proof}
By definition,
$H^0(S,R^2f_{\crys*}\cO_{X/W_n})=\Hom(\cO_{S/W},R^2f_{\crys*}\cO_{X/W_n}).$
By Lemma \ref{lemmaflatnessofhigherimage},
$R^2f_{\crys*}\cO_{X/W_n}$ is a coherent crystal.
Since $C(k) \lra W(\bar{k})$ is faithfully flat, the pull-back functor from the category of crystal on $S$ to the one of $\bar{S}$ is faithful. By the base change theorem of crystalline cohomology, 
$R^2f_{\crys*}\cO_{\overline{X}/W_n}$
is naturally isomorphic to the pull-back of 
$R^2f_{\crys*}\cO_{X/W_n}$. This proves the claim.
 \end{proof}
\begin{cor}\label{beta5injective}
The natural map
$$H^0_{\syn}(S,R^2f_{\syn*}\cO_{X,n}^{\crys})\lra H^0_{\syn}(\bar{S},R^2f_{\syn*}\cO_{\overline{X},n}^{\crys})
$$
is injective. As a result, the natural map
$$H^0_{\syn}(S,R^2f_{\syn*}\cI_n)^{1=p^{-1}F}\lra H^0_{\syn}(\bar{S},R^2f_{\syn*}\cI_n)^{1=p^{-1}F}
$$
is al. injective uniformly in $n$.
\end{cor}
\begin{proof}
By the proof of Proposition \ref{propositionalmostcrystal}, the natural map
$$
H^0_{\syn}(S,R^2f_{\syn*}\cO_{X,n}^{\crys}) \lra H^0(S,R^2f_{\crys*}\cO_{X/W_n})
$$
is actually an isomorphism since $R^if_{\CRYS*}\cO_{X/W_n}$ ($i=0,1$) is a coherent crystal over $\SYN(S)$. Then, the first claim follows from Lemma \ref{keyinjective}.
The second claim follows from the fact that the natural map
$$\cI_n \lra \cO_{n}^{\crys}$$
has a kernel and cokernel killed by $p$.
\end{proof}
 \begin{lem}\label{beta3iso}
  The natural map  $$H^0(\overline{S},R^2f_{\syn*}\mu_{p^n,X_{\overline{k}}})\lra H^0_{\syn}(\overline{S},R^2f_{syn*}\cI_n)^{1=p^{-1}F} $$
  is surjective and al, injective uniformly in $n$.
 \end{lem}
 \begin{proof}
Since any syntomic cover of $\overline{S}$ has a section, so $H^i_{\syn}(\overline{S}, \cF)=0$ for any sheaf $\cF$ on $\syn(\overline{S})$ and $i>0$. By the Leray spectral sequence, for any sheaf $\cG$ on $\syn(\overline{X})$, there is a canonical isomorphism
$$H^i_{\syn}(\overline{X},\cG)\cong H^0_{\syn}(\overline{S},R^if_{\syn*}\cG).$$
Take cohomology on $\overline{X}$ for the exact sequence in Proposition \ref{Bauerexact}. This shows the surjectivity in the claim. The natural map
$$H^0_{\syn}(\overline{S},R^2f_{\syn*}\cI_n)^{1=p^{-1}F} \lra H^0(\overline{S},R^2f_{\crys*}\cO_{\overline{X}/W_n})^{F=p}=H^2_{\crys}(\overline{X}/W_n)^{F=p}$$
is an al. isomorphism. Since $\Spec(\overline{W})$ is a final object in $\CRYS(\overline{S}/W)$ \cite[II. 2.2]{Kat},  $H^2_{\crys}(\overline{X}/W_n)=H^2_{\crys}(\overline{X}/\overline{W_n})$. By the trace map of crystalline cohomology,
$$H^2_{\crys}(\overline{X}/\overline{W_n})^{F=p}=(\overline{W_n})^{F=id}=\bZ/p^n.
$$
Considering Lemma \ref{lemmainjofflat}, the surjectivity implies that the map is al. injecitive uniformly in $n$.
\end{proof}

\subsection{Albanese variety}
 Assume that $X$ has a $k$-rational point $a$.
 $a$ induces an Albanese map $a:X\to A:=Alb(X)$ be the map. Proposition \ref{propositionrelateAlbanese} and its proof implies $a^*$ induces an al. isomorphism uniformly in $n$
 $R^1f_{syn*}\cO_{X,n}^{crys}\cong R^1f_{syn*}\cO_{A,n}^{crys}$ and the result holds for $\cI_n$ as $\cI_n$ is al. isomorphic to $\cO_n^{crys}$ uniformly in $n$. All our construction are functorial, so $\coker(g_{1,A})\cong \coker(g_{1,X})$ is  al. uniformly in $n$.

 \subsection{General $X$}
 For general $X$, we use Grothendieck's trick: We choose a series of hyperplane sections to get $Y\subset X$ such that $Y$ is a smooth projective geometrically connected curve.  To see the al. vanishing of $\gamma_X$ uniformly in $n$, we have the following diagram 
 $$\begin{CD}
 	H^0(S,\ker(g_{2,X}))@>\gamma_X>> H^1(S,\coker(g_{1,X}))\\
 	@V\beta_1VV @VV\beta_2V\\
 	H^0(S,\ker(g_{2,Y}))@>\gamma_Y>> H^1(S,\coker(g_{1,Y})).
 \end{CD}$$
 We already know $\gamma_Y$ vanishes so it is enough to see $\beta_2$ is injective. For simplicity, we assume that $Y$ has a rational $k$-point. The universal property of Albanese variety induces a diagram 
 $$\begin{CD}
 	Y@>>> X\\
 	@VVV @VVV\\
 	A_Y@>>> A_X.
 \end{CD}$$
The natural map $A_Y\to A_X$ is surjective. This is equivalent to show that the natural map between the associated Picard varities has a finite kernel. Since $Y$ is obatined by taking hyperplanes section, 
$$
H^1_{et}(X_{k^s},\bQ_\ell(1)) \lra H^1_{et}(Y_{k^s},\bQ_\ell(1))
$$
is injective. So the  map between the Tate-modules of the associated Picard varieties is injective. By the Poincar\'e reducibility theorem, $A_Y\cong A_X\times B$ for another abelian variety $B$ up to an isogeny. Roughly speaking,
$\coker(g_{1,A_X})$ is almost a direct summand of the sheaf $\coker(g_{1,A_Y})$ uniformly in $n$. So the injectivity follows from the comparison between $X$ and $A_X$. This proves that  $\gamma_X$ is al. 0 uniformly in $n$.
The almost vanishing of $H^0(S,\coker(g_1))$ follows from a similar argument.

Now we need to remove the assumption of existence of a rational $k$-point. By taking a finite Galois extension $l$ of $k$, we can assume that the claim is true for $X_l/l$. Then the claim for $X$ follows from the Galois \'etale descent. This proves that the natural map
$$H^0(S,R^2f_{\syn*}\mu_{p^n})\to H^0(S, R^2f_{\syn*}\cI_n)^{1=p^{-1}F}$$
is an al. isomorphism uniformly in $n$. Composites with the natural map
$$H^0(S, R^2f_{\syn*}\cI_n)^{1=p^{-1}F} \lra H^0_{\cris}(S,R^2f_{\crys*}\cO_{X/W_n})^{F=p}.$$
We get an al. isomorphism uniformly in $n$
$$H^0(S,R^2f_{\syn*}\mu_{p^n})\to H^0_{\cris}(S,R^2f_{\crys*}\cO_{X/W_n})^{F=p}$$
By the Lemma below and invert $p$, we get an isomorphism
 $$(\varprojlim_n H_{\syn}^0(S,R^2f_{\syn*}\mu_{p^n}))\otimes_{\bZ_p} \bQ_p\stackrel{\cong} \longrightarrow H^0_{\crys}(S,R^2f_{\crys*}\cO_{X/W})^{F=p}[1/p].$$
This completes the proof of Proposition \ref{theoremtowardscrystalline}.
\begin{lem}\label{gotorational}
  The natural morphism 
  $$
  H^0_{\cris}(S,R^2f_{crys*}\cO_{X/W}) \lra \varprojlim_n H^0_{\cris}(S,R^2f_{\crys*}\cO_{X/W_n}) 
  $$
  is an al. isomorphism.
\end{lem}
\begin{proof}
By Lemma \ref{takinglimit}, it suffices to prove that
$$
  H^0_{\cris}(S,\cE_{X/S}) \lra \varprojlim_n H^0_{\cris}(S,i_n^*\cE_{X/S}) 
$$
is an isomorphism. This follows from  \cite[7.22.2]{BerthelotOgus1978} since $\cE_{X/S}$ is a coherent crystal on $S/W$.
  \end{proof}

\section{Proof of Proposition \ref{crystallinegenerictorigid}}
\subsection{F-crystals}
 We still need to see 
 $$H^0_{\crys}(S,R^2f_{\crys*}\cO_{X/W})^{F=p}[1/p]\cong H^0_{\conv}(U,R^2f_*\cO_X)^{F=p}$$

Let $S$ be a Noetherian scheme of characteristic $p$. An $F$-crystal over $S$ is an $\cO_{S/\bZ_p}$-coherent crystal $\cE$ with a morphism $F:\cE^{(\sigma)}\to \cE$ such that the kernel and cokernel of $F$ are annihilated by a power of $p$, where $\cO_{S/\bZ_p}$ is the structure sheaf of the big crystalline site $S$ over $\bZ_p$. The Morphisms are morphisms of coherent crystals compatible with $F$. The category of $F$-isocrystals is the isogeny category of $F$-crystals. We denote the category of $F$-crystals via $\Fcris (S/W)$ and $F$-isocrystals via $\Fisoc(S/W)$. If $S$ is of finite type, by \cite[Theorem 2.4.2]{berthelot1996cohomologie}, our category of $F$-isocrystals can be naturally seen as a full subcategory of convergent $F$-isocrystals and if $S$ is smooth over $\bF_p$, it is an equivalence of categories up to Tate twists. \\
\begin{lem}\label{abeliancate}
If $S$ is the spectrum of a field of characteristic $p>0$ or an affine smooth variety over a perfect field $k$ of characteristic $p>0$. Then $\Fcris(S/W)$ is an abelian category and the functor from $\Fcris(S/W)$ to the category of coherent crystals on $(S/W)$ is exact. Moreover, if $S$ is an affine smooth variety over a perfect field $k$, the functor from $\Fcris(S/W)$ to the category of convergent $F$-isocrystals is exact.
\end{lem}
\begin{proof}
By \cite[Prop. IV.1.7.6]{berthelot2006cohomologie}, the category of coherent crystals is an abelian category. It suffices to show that the kernel and cokernel inherit a non-degenerated $F$-structure.  Let $F_S$ denote the absolute Frobenus on $S$. The cokernel is clear since $F_S^*$ is a right exact functor. Let $0\to \cF_1\to \cF_2\to \cF_3\to 0$ be an exact sequence is an exact sequence in the category of crystals. Assume that the third arrow is an morhpism of $F$-crystals. The pull-back gives an exact sequence of crystals
$$
F_S^*\cF_1\to F_S^*\cF_2\to F_S^*\cF_3\to 0
$$
Next, we show that the surjective morphism
$$
F_S^*\cF_1\to \ker(F_S^*\cF_2\to F_S^*\cF_3)
$$
is al. injective. Write $S=\Spec(A)$. Let $\cA$ be an p-adic comolete noetherian algebra such that $A=\cA/p$ (cf. \cite[Lem. A.3 and Cor. A.5]{Mor} for details). Let $M_i$ ($i=1,2,3$) denote the corresponding $\cA$-modules of $\cF_i$. By \cite[Cor. A.5]{Mor}, each $M_i[1/p]$ is a finite projective $\cA$-module. So the exact sequence 
$$
0\to M_1[1/p]\to M_2[1/p]\to M_3[1/p]\to 0
$$
split as $A[1/p]$-modules. Let $\sigma:\cA \lra \cA$ be a lifting of the absolute Frobenius $F_S$. So the sequence 
$$
0\to (M_1\otimes_{\cA,\sigma}\cA)[1/p]\to (M_2\otimes_{\cA,\sigma}\cA)[1/p]\to (M_3\otimes_{\cA,\sigma}\cA)[1/p]\to 0
$$
is exact. So 
$$
F_S^*\cF_1\to \ker(F_S^*\cF_2\to F_S^*\cF_3)
$$
is an isomorphism after tensor $\bQ_p$. Its kernel is killed by a power of $p$ since the $\cA$-module of the kernel is finitely generated. Then, by digaram chasing, the induced morphism 
$$F_S^*\cF_1\to \cF_1$$
is an isogeny. It is easy to check by diagram chasings that $\cF_1$ is the $\ker(F_S^*\cF_2\to F_S^*\cF_3)$ in $\Fcris(S/W)$. This proves the first claim. For the second claim, by definition \cite[Def. 2.1]{kedlaya2016notes}, the category of convergent $F$-isocrystals is equivalent to the category of coherent $\cD$-module over $\cA[1/p]$ with a semi-linear $\sigma$-action. The claim follows from the fact that an exact sequence of $\cA$-modules is an exact sequence of $\cA[1/p]$-modules after tensor $\bQ_p$.
\end{proof}

\begin{lem}\label{toFcrystal}
Let $f:\cX \to U$ be a smooth projective morphism with geometrically connected fibers. Assume that $U$ is a smooth affine integral variety over $\bF_p$. Then, there is an open dense subset $V\subseteq U$, a $F$-crystal $\cF \in \Fcris(V/W)$ and a $F$-morphism $\cF \to R^2f_{\crys*}\cO_{\cX/W}|_V$ whose kernel and cokernel are killed by a power of $p$. Moreover, the induced morphism
$$
H^0_{\cris}(V,\cF)^{F=p}[p^{-1}]\cong H^0_{\cris}(U,R^2f_{\crys*}\cO_{\cX/W})^{F=p}[p^{-1}]
$$
is an isomorphism.
\end{lem}
\begin{proof}
By \cite[Lem. A.2]{Mor}, there is a coherent crystal $\cE_{\cX/U}$ and a morphism of $\cO_{U/W}$-module
$\alpha:\cE_{\cX/U} \to R^2f_{\crys*}\cO_{\cX/W} $ whose kernel and cokernel are killed by a power of $p$. So, there is a morphism of $\cO_{U/W}$-module
$$\beta: R^2f_{\crys*}\cO_{\cX/W} \to \cE_{\cX/U}$$
such that $\alpha\beta =\beta\alpha=p^n$. Let $a: F^*_{U}R^2f_{\crys*}\cO_{\cX/W} \to R^2f_{\crys*}\cO_{\cX/W} $ be the morphism defined in \cite[Prop. A.7 (2)]{Mor}. Define $b:=\beta a F^*_U(\alpha)$. 
$$\alpha b=\alpha\beta aF^*_U(\alpha)=p^naF^*_U(\alpha)=p^n aF^*_U(\alpha)$$
Let $R^2f_{\crys*}\cO_{\cX/W}(-n)$ denote $p^n a:F^*_{U}R^2f_{crys*}\cO_{\cX/W} \to R^2f_{\crys*}\cO_{\cX/W}$. $(\cE, b)$ is a $F$-crystal isogenous to $R^2f_{\crys*}\cO_{\cX/W}(-n)$. By \cite[Prop. A.8]{Mor}, all slopes of $(\cE, b)$ are $\geq n$. View $\cE(n)$ as an object in the category of convergent $F$-isocrystals. By \cite[Lem. 5.8]{KrPa}, there is an open dense open subset $V$ of $U$ such that $\cE(n)|_V$ is isomorphic to a $F$-crystal $(\cF, c)$ in the category of convergent $F$-isocrystals on $V$. So $\cE|_V$ is isomorphic to $\cF(-n)$ in the category of convergent $F$-isocrystals on $V$. Hence, there is an $F$-isogeny $\gamma: \cF(-n) \to \cE|_V$ in $\Fcris(V/W)$ ( It is also an isogeny in the category of $\cO_{V/W}$-modules) . By definition, $\cF(-n)=(\cF, p^nc)$ and $\gamma:\cF \to \cE|_V$ is a morphism of crystals on $V$. The composition $\eta:=\alpha_V \gamma$ is a $F$-morphism $\cF(-n)\to R^2f_{\crys*}\cO_{\cX/W}(-n)|_V$, where $\alpha_V:=\alpha|_V$. Thus, $(p^na|_V)F_V^*(\eta)=\eta (p^nc)$. This is same as $a|_V F_V^*(p^n\eta)=(p^n\eta)c$. Thus, $p^n\eta$ is a $F$-morphism $\cF\to R^2f_{\crys*}\cO_{\cX/W}|_V$ and is also an isogeny in the category of $\cO_{V/W}$-modules since both of $\alpha_V$ and $\gamma$ are isogenies in the category of $\cO_{V/W}$-modules. This proves the first claim.

Next, we prove the second claim. By definition,
\begin{align*}
    H^0_{\cris}(U,R^2f_{\crys*}\cO_{\cX/W})^{F=p}[p^{-1}]&=H^0_{\cris}(U,R^2f_{\crys*}\cO_{\cX/W}(-n))^{F=p^{n+1}}[p^{-1}]\\
    =H^0_{\cris}(U,\cE)^{F=p^{n+1}}[p^{-1}]&=\Hom_{\Fisoc}(\cO_{U/W}(-(n+1)),\cE).
\end{align*}
By \cite[Thm. 5.3]{kedlaya2016notes}, 
\begin{align*}
    &\Hom_{\Fisoc}(\cO_{U/W}(-(n+1)),\cE)\cong \Hom_{\Fisoc}(\cO_{V/W}(-(n+1)),\cE|_V)\\
    &\cong  \Hom_{\Fisoc}(\cO_{V/W}(-(n+1)),\cF(-n))
    \cong  \Hom_{\Fisoc}(\cO_{V/W}(-1),\cF)\\
    &\cong   H^0_{\cris}(V, \cF)^{F=p}[p^{-1}].
\end{align*}
\end{proof}

\begin{prop}\label{thelastprop}
Let $f:\cX \to U$ as in the previous lemma. Write $g: S=\Spec(K) \to U$ for the generic point of $U$ and $X$ for the generic fiber.
Then there is a canonical isomorphism induced by $g$
$$H^0_{\cris}(U,R^2f_{\crys*}\cO_{\cX/W})^{F=p}[p^{-1}] \cong
H^0_{\cris}(S,R^2f_{\crys*}\cO_{X/W})^{F=p}[1/p].
$$
\end{prop}
\begin{proof}
By Lemma \ref{localization} below, the $F$-morphism
$$g^*_{\crys}R^nf_{\crys*}\cO_{\cX/W}\longrightarrow R^nf_{\crys*}\cO_{X/W}$$
has a kernel and a cokernel killed by a $p$-power. Let $\cF$ as in the previous lemma, $\cF|_S \lra g^*_{\crys}R^nf_{\crys*}\cO_{\cX/W}$
is an $F$-morphism with a kernel and a cokernel killed by a $p$-power. This gives
$$H^0_{\cris}(S,\cF|_S)^{F=p}[p^{-1}]\cong H^0_{\cris}(S,R^2f_{\crys*}\cO_{X/W})^{F=p}[1/p].$$
By the second claim of the previous lemma, it remains to prove that the pull-back induces a canonical isomorphism
$$
H^0_{\cris}(V,\cF|_S)^{F=p}[p^{-1}]\cong H^0_{\cris}(S,\cF)^{F=p}[p^{-1}].
$$
This is proved in Lemma \ref{thelastlem} below.
\end{proof}

\begin{lem}\label{localization}
    Consider the following Catesian diagram 
    $$\begin{CD}
        X@>g'>>\cX\\
        @Vf'VV @VVfV\\
        S@>>g> U
    \end{CD}$$
    There exists a natural $F$-morphism $$g^*_{\crys}R^nf_{\crys*}\cO_{\cX/W}\longrightarrow R^nf'_{\crys*}\cO_{X/W}$$
    which is an isomorphism up to a finite $p$-exponent. 
\end{lem}

\begin{proof}
    By smooth base change, we have a map 
    $$g^*_{\crys}R^nf_{\crys*}\cO_{\cX/W}\to H^n(Lg^*_{\crys}Rf_{\crys*}\cO_{\cX/W})\stackrel{\cong}\longrightarrow H^n(Rf'_{\crys*}\cO_{X/W})=R^nf'_{\crys*}\cO_{X/W}.$$
    This map is $F$-equivariant because both $F$-structures are induced from the absolute and relative Frobenius whose formation are functorial (see \cite[Prop. A.7]{Mor}). To see the first map is an isomorphism up to a finite $p$-exponent, take a $p$-adically complete, noetherian flat lifting $R$ of $U$ to $W$. Since $R/p^nR$ is quasi-smooth over $W_n$ for all $n$, it suffices to check on each affine $(V,T=\Spec B,\delta)\in Crys(S/W)$ such that there exists a map $v_T:R\to B$ lifting $V\to S$ (after a easy covering argument). We need to see for $i\geq 1$, $L^iv_T^*H^{n+i}_{\crys}(\cX/R)$ is killed by a power of $p$ independent of $v_T$. This follows from \cite[Coro. A.5]{Mor} that $H^{n+i}_{\crys}(\cX/R)$ is a finitely generated $R$-module which becomes finite projective over $R[1/p]$ after inverting $p$ (cf. \cite[Proof of Lem. A.2]{Mor}).
\end{proof}

\begin{lem}\label{thelastlem}
Notations as in Lemma \ref{toFcrystal},
$$
H^0_{\cris}(V,\cF)^{F=p}[p^{-1}]\cong H^0_{\cris}(S,\cF|_S)^{F=p}[p^{-1}].
$$
\end{lem}
\begin{proof}
By shrinking $V$, we can assume that $V$ is affine and the slope polygon of $\cF$ is constant on $V$ (cf. \cite[Thm. 3.12]{kedlaya2016notes}). By \cite[Corollary 4.2]{kedlaya2016notes}, there exists a unique filtration in the category of convergent $F$-isocrystals
$$0=\cE_0\subset \cE_1\subset \cdots \cE_l=\cF_{\bQ_p}$$
such that each successive quotient $\cE_i/\cE_{i-1}$ is isoclinic of slope $s_i\in \bQ_{\geq 0}$ and $s_1<s_2<...<s_l$. Since all slope of $\cF$ are nonnegative.
There is a $\cE_i$ such that all of slopes of $\cE_i$ are contained in $[0,1]$ and all of slopes of $\cF_{\bQ_p}/\cE_i$ are $>1$. Denote it by $\cF_{\bQ_p,[0,1]}$. By shrinking $V$ further, we can assume that $\cF_{\bQ_p,[0,1]}$ comes from a $F$-crystal $\cF_{[0,1]}$ (cf. \cite[Lem. 5.8]{KrPa}). There is a morphism $\cF_{[0,1]} \to \cF$ in $\Fcris(V/W)$ which becomes injective in $\Fisoc(V/W)$. By Lemma \ref{abeliancate} replacing $\cF_{[0,1]}$ by its image in $\cF$, we can assume that $\cF_{[0,1]} \to \cF$ in $\Fcris(V/W)$ is injective. Let $\cF_{>1}$ denote the cokernel which has all slope greater than $1$ since restricting to a closed point is exact (cf. \cite[Lem. A.9]{Mor}). Let $\cA$ be a smooth lifting of $V$ to $W$ and $\cB$ denote the local ring of the generic point of the special fiber of $\Spec(\cA)$.  It's easy to see that $ \cB$ is a DVR with the maixmal ideal generated by $p$. Since $\cA \to \cB$ is flat, it is still flat after taking completion at $p$ (cf. \cite[\href{https://stacks.math.columbia.edu/tag/0AGW}{Tag 0AGW}]{stacks-project}). By abusing notations, we still denote them by $\cA$ and $\cB$ respectively. Hence, the pull-back functor $\Fcris(V/W) \to \Fcris(S/W)$ is exact. Thus, there is an exact sequence in $\Fcris(S/W)$
$$0\to\cF_{[0,1]} |_S\to \cF |_S\to \cF_{>1}|_S \to 0.$$
This gives
$$0\to \Hom_{\Fcris}(\cO_{S/W}(-1),\cF_{[0,1]}|_S) \to \Hom_{\Fcris}(\cO_{S/W}(-1),\cF |_S)\to \Hom_{\Fcris}(\cO_{S/W}(-1), \cF_{>1}|_S)
$$
The last group is of finite exponent, since the pull-back functor from $\Fcris(S/W)$ to $\Fcris(\Spec(\bar{K})/W)$ is faithful and all slopes of $\cF_{>1}|_S$ are greater than 1. Thus,
$$\Hom_{\Fcris}(\cO_{S/W}(-1),\cF_{[0,1]}|_S)[p^{-1}] \cong \Hom_{\Fcris}(\cO_{S/W}(-1),\cF |_S)[p^{-1}]
$$
Similarly (\cite[Lem. A.9]{Mor}),
$$\Hom_{\Fcris}(\cO_{V/W}(-1),\cF_{[0,1]})[p^{-1}] \cong \Hom_{\Fcris}(\cO_{V/W}(-1),\cF)[p^{-1}].
$$
Thus, we may assume that $\cF_{[0,1]}=\cF$. By shrinking $V$ (cf.\cite[Lem. 5.8]{KrPa}), we may assume that $\cF$ is $F$-isogenous to a Dieudonn\'e crystal. In fact, we can assume that $\cF$ is a Dieudonn\'e crystal. Since the functor from the category of Dieudonn\'e crystals to $\Fcris$ on $S$ or $V$ is fully faithful (cf. \cite[Rmk. 5.2]{KrPa}), it suffices to show
$$\Hom_{DC(V)}(\cO_{V/W}(-1),\cF)[p^{-1}] \cong \Hom_{DC(S)}(\cO_{S/W}(-1),\cF|_S)[p^{-1}].
$$
This follows from the equivalences between Dieudonne crystals and $p$-divisble groups on $S$ and on $V$ and the fact that the restriction functor for $p$-divisible groups from $V$ to $S$ is fully faithful (cf. \cite{de1995crystalline,de1998homomorphisms}). We note that the fully faithfulness on a normal base is reduced to DVR in \cite[Sec. 1]{messing1982theorie}, which is proved in $loc.cit$.
\end{proof}

\subsection{Proof of Theorem \ref{thekeythm}}
It remains to show that
 $$H^0_{\crys}(U,R^2f_{\crys*}\cO_{\cX/W})^{F=p}[1/p]\cong H^0_{\conv}(U,R^2f_*\cO_{\cX/K})^{F=p}$$
It follows from \cite[Cor. A.13]{Mor} that the F-isocrystal
$R^2f_{\crys*}\cO_{\cX/W}\otimes\bQ_p$ is naturally isomorphic to the convergent F-isocrystal $R^2f_*\cO_{\cX/K}$
constructed by Ogus \cite[Thm. 3.1]{ogus1984f}. This completes the proof of Theorem \ref{thekeythm}.

\section{Equivalence with D'Addezio's Conjecture}
Let $X$ be a smooth projective geometrically connected variety over a finitely generated field $k$ of characteristic $p>0$. In \cite{DAd}, D'Addezio defined $H^2_{\fppf}(X_{\bar{k}},\bQ_p(1))^k$ as the image of 
$$H^2_{\fppf}(X,\bQ_p(1))\longrightarrow H^2_{\fppf}(X_{\overline{k}},\bQ_p(1)),$$
and he gave a formulation of the $p$-adic Tate conjecture for divisors:
\begin{conj}[D'Addezio]
The cycle class map 
$$\Pic(X)\otimes_\bZ\bQ_p \lra H^2_{\fppf}(X_{\bar{k}},\bQ_p(1))^k$$ is surjective.
\end{conj}

\begin{thm}\label{theoremequivalence}
The above conjecture is equivalent to the finiteness of the exponent of $\Br(X_{k^s})^{G_k}[p^\infty]$.
\end{thm}

\begin{lem}\label{lemmarationales}
Let $X$ be a smooth projective geometrically connected variety over a field $k$. The sequence from the Kummer theory
\begin{equation}\label{formulacanonicalks}
0\lra (\Pic(X_{\overline{k}})\widehat{\otimes} \bQ_p)^{G_k}\lra H^2_{\fppf}(X_{\overline{k}},\bQ_p(1))^{G_k} \lra V_p\Br(X_{\overline{k}})^{G_k}\lra 0 \end{equation}
is exact.
\end{lem}
\begin{proof}
It suffices to show the right exactness. There is a long exact sequence
$$0\lra (\Pic(X_{\overline{k}})\widehat{\otimes} \bQ_p)^{G_k}\lra H^2_{\fppf}(X_{\overline{k}},\bQ_p(1))^{G_k} \lra V_p\Br(X_{\overline{k}})^{G_k}\lra H^1(G_k, \Pic(X_{\overline{k}})\widehat{\otimes} \bQ_p).$$
It remains to show the last arrow vanishes. Note that if we
replace $X$ by a curve, the last arrow vanishes. Thus, by the pull-back trick, it suffices to find finitely many smooth projective geometrically connected curves $Y_i\subseteq X$ ($1\leq i\leq r$) such that the induced natural map
 $$H^1(G_k, \Pic(X_{\overline{k}})\widehat{\otimes} \bQ_p)\lra  \prod_i^r H^1({G_k, \Pic(Y_{\overline{k}})\widehat{\otimes} \bQ_p})$$
 is injective. Since $$\Pic(X_{\bar{k}})\widehat{\otimes} \bQ_p\cong \NS(X_{\bar{k}})\otimes_\bZ \bQ_p\cong\NS(X_{k^s})\otimes_\bZ \bQ_p$$
it suffices to choose $Y_i$ such that $\NS(X_{\overline{k}})\otimes_{\bZ}\bQ_p \lra \prod^r_i \NS(Y_{i,\overline{k}})\otimes_{\bZ} \bQ_p$ is injective (because then it splits as the $G_k$-action factors through a finite quotient). Let $l/k$ be a finite Galois extension such that $\NS(X_l)\lra \NS(X_{k^s})$ is surjective. We first prove the claim for $X_l/l$. This follows from \cite[p17]{Yua2} that there exist curves $Y_i\subseteq X_l$ such that $\NS(X_{\overline{k}})\lra \prod_i^r NS(Y_{i,\overline{k}})$ has a kernel and a cokernel of finite exponent. So there is an exact sequence
$$0\lra (\Pic(X_{\overline{k}})\widehat{\otimes} \bQ_p)^{G_l}\lra H^2_{\fppf}(X_{\overline{k}},\bQ_p(1))^{G_l} \lra V_p\Br(X_{\overline{k}})^{G_l}\lra 0.$$
Take $\Gal(l/k)$-invariant of the above exact sequence, we get desired result since 
$$H^1(\Gal(l/k),(NS(X_{\overline{k}})\otimes_{\bZ} \bQ_p)^{G_l})=0.$$ 
\end{proof}

\begin{lem}\label{lemmasurj}
    The map $V_p\Br(X) \lra V_p\Br(X_{k^s})^{G_k}$ is surjective.
\end{lem}

\begin{proof}
 Write $M_{X,n}:=\Im(\Br(X)[p^n]\lra \Br(X_{k^s})^{G_k}[p^n])$ and $N_X:=\Ker(\Br(X)\lra \Br(X_{k^s})^{G_k}$. By Lemma \ref{bigtrick}, the map $\varprojlim_n M_{X,n}\lra  T_p(\Br(X_{k^s})^{G_k}$ has a cokernel of finite exponent. By the exact sequence
 $$0\lra  T_p(N_X)\lra T_p(\Br(X))\lra \varprojlim_n M_{X,n}\stackrel{\delta}\lra R^1\varprojlim_n N_X[p^n] $$
 it suffices to show $\delta$ has a image of finite exponent. Let $Y\subseteq X$ be a smooth projective geometrically integral curve over $k$ obtained by hyperplane sections repeatedly. By the same argument as in proof of Lemma \ref{bigtrick} to enlarge $k$, we get $N_X[p^n]=\Br(k)[p^n]\oplus H^1(k,\Pic_{X/k})[p^n]$ and $N_Y[p^n]=\Br(k)[p^n]\oplus H^1(k,\Pic_{Y/k})[p^n]$. Since $\Br(Y_{k^s})=0$, it is enough to show $R^1\varprojlim_n H^1(k,\Pic_{X/k})[p^n]\lra R^1\varprojlim_n H^1(k,\Pic_{Y/k})[p^n]$ is injective. The result follows from the proof of Lemma \ref{bigtrick}, that the natural map $H^1(k,\Pic^0_{X/k})\lra H^1(k,\Pic_{X/k})$ is al. isomorphism and that there is ab abelian group $D$ together with a morphism $D\lra H^1(k,\Pic_{X/k}^0)$ such that the induced map $H^1(k,\Pic_{X/k}^0)\oplus D\to H^1(k,\Pic^0_{Y/k})$ is an al. isomorphism.
\end{proof}

\begin{proof}[Proof of theorem \ref{theoremequivalence}]
Consider the commutative diagram
\begin{displaymath}
\xymatrix{
0\ar[r]& (\Pic(X)\otimes \bQ_p)\ar[r]\ar[d] &H^2_{\fppf}(X,\bQ_p(1))\ar[r]\ar[d]& V_p\Br(X)\ar[r]\ar[d]& 0 \\
0\ar[r]&(\Pic(X_{\bar{k}})\widehat{\otimes} \bQ_p)^{G_k}\ar[r]&H^2_{\fppf}(X_{\bar{k}},\bQ_p(1))^{G_k} \ar[r]& V_p\Br(X_{\bar{k}})^{G_k}\ar[r]& 0}  
\end{displaymath}
Since $\Pic(X_{\bar{k}})\widehat{\otimes} \bQ_p\cong\NS(X_{k^s})\otimes_\bZ\bQ_p$, the image of the first column is identified with $\NS(X)\otimes_\bZ\bQ_p$. It's easy to see 
 $\NS(X)\otimes_\bZ\bQ_p\cong (\NS(X_{k^s})\otimes_\bZ\bQ_p)^{G_k}$. Hence, the first column is surjective. By Lemma \ref{lemmasurj}, the image of the third arrow coincides with the image of the injective map $ V_p\Br(X_{k^s})^{G_k}\lra V_p\Br(X_{\bar{k}})^{G_k}$. By diagram chasing, the equivalence between the two formulations of the $p$-adic Tate conjecture follows from proposition \ref{firstexact}.
\end{proof}

\section{Weak Lefschetz theorem for Brauer Groups}
In this section, we prove Theorem \ref{mainlef}. We first recall two lemmas on Brauer groups.

\begin{lem}\label{lemmareduce}
    Let $X$ be a smooth projective geometrically integral variety over a field $k$ of dimension $\geq 3$. Let $Y\subseteq X$ be a smooth hyperplane section of $X$. Then, to show that the pullback map $\Br(X)\to \Br(Y)$ has a kernel and cokernel of finite exponent, it suffices to show that for all primes $l'$, the pullback map $V_{l'}(\Br(X_{k^s})^{G_k})\to V_{l'}(\Br(Y_{k^s})^{G_k})$ is an isomorphism, and for all but finitely many $\ell'$, the pullback map $\Br(X_{k^s})^{G_k}[l'^\infty]\to \Br(Y_{k^s})^{G_k}[l'^\infty]$ is an isomorphism.
\end{lem}
\begin{proof}

 By \cite[Lem 3.1]{Yua2}, there is a functorial exact sequence in the isogeny category of abelian groups.
\begin{equation}\label{formulaarithgeo}
 0\longrightarrow H^1(k,\Pic(X_{k^s}))\longrightarrow \Br(X)/\Br(k)\longrightarrow \Br(X_{k^s})^{G_k}\longrightarrow 0.\end{equation}
 The inclusion induces a morphism from the exact sequence \eqref{formulaarithgeo} of $X$ to that of $Y$. By the weak Lefschetz theorem of $l$-adic cohomology and isogeny theorem (cf. \cite{Zar1}), the induced map $\Pic^0_{X/k,\red}\to \Pic^0_{Y/k,\red}$ is an isogeny, so is the corresponding map $H^1(k,\Pic(X_{k^s}))\to H^1(k,\Pic(Y_{k^s}))$. Thus, to show that the map $\Br(X)\to \Br(Y)$ has a kernel and cokernel of finite exponent, it suffices to show it for the map $\Br(X_{k^s})^{G_k}\to \Br(Y_{k^s})^{G_k}$. The claim for $\Br(X_{k^s})^{G_k}\to \Br(Y_{k^s})^{G_k}$ is equivalent to that $\Br(X_{k^s})^{G_k}[l'^\infty]\to \Br(Y_{k^s})^{G_k}[l'^\infty]$ has a kernel and cokernel of finite exponents and is an isomorphism for almost all $l'$.  Since $\Br(X_{k^s})^{G_k}$ and $\Br(Y_{k^s})^{G_k}$ are isogenous to torsion abelian groups of cofinite type, this is equivalent to the condition in the lemma.
\end{proof}

So to prove Theorem \ref{mainlef}, it is enough to check the condition in Lemma \ref{lemmareduce}, where the condition on rational Tate module is proved in Proposition \ref{proprational} and the condition on torsion groups for sufficiently large primes is proved in Proposition \ref{propintegral}.

\begin{lem}\label{lemmaglobalapproximation}
 	Let $X$ be a smooth projective geometrically integral variety over a field k. Then for primes $l\neq \mathrm{char}(k)$ sufficiently large, the sequence from the Kummer theory
 	\begin{equation}\label{formulaexact2}
 	 0\longrightarrow  \NS(X)/l^n\NS(X)\longrightarrow H^2(X_{k^s},\mu_{l^n})^{G_k}\longrightarrow \Br(X_{k^s})^{G_k}[l^n]\longrightarrow 0 
 	 \end{equation}
 	are exact for any $n\geq 1$.
 	 \end{lem}
 	 
 	 \begin{proof} 
     cf. \cite[Thm. 5.3.1]{CTS2} or \cite[Prop. 2.1]{Qin3}.
 	 \end{proof}

\subsection{Rational result}
Fixed a prime $l\neq p$. We choose an algebraic closure $\bF$ of $\bF_p$ and all finite field extensions of $\bF_p$ are seen as subfield of $\bF$.

\begin{lem}\label{lemmaweaklef1} Let $f:Y\to X$ be a smooth projective morphism between smooth integral varieties over $\bF_p$ with a geometrically connected generic fiber. Let $p_1:X'\to X$ be a connected finite etale Galois covering map.
    Then there is a sparse subset $S\subseteq |X|$ such that for any $z\in X(k)\setminus S $, $\pi_1^{-1}(z)$ consists of a single point $z'$, which is geometrically irreducible over $k_0':=k(X')\cap \bF$ and that $\pi_1(z')$ and $\pi_1(X')$ have the same image under 
    $$\pi_1(z')\to \pi_1(X')\to \GL(H^2(Y_{\overline{\eta}},\bQ_l(1))). $$
\end{lem}

\begin{proof}
    Let $d:=[K(X'):K(X)]$, $k_0':=K(X')\cap \bF$ and $k_0''$ be a finite extension of $k_0'$ such that for any finite extension $l/k_0'$ of degree $\leq d$, $l\subseteq k_0''$. Put $X'':=X'\times_{\Spec k_0'} \Spec k_0''$ and $p_2$ be the projection $p_2: X''\to X'\to X$. Let $S_1$ be the set $S_1:=\{x\in |X|: p_2^{-1}(x) \text{ contains more than one closed point} \}$. So, if $z\in |X|\setminus S_1$, $p_1^{-1}(z)=z'$ is a reduced closed point. We say $z'\subseteq |X'|$ is $l$-strictly Galois generic if the image of $\pi_1(z',\overline{z'})$ and $\pi_1(X',\overline{z'})$ under 
    $$\pi_1(z',\overline{z'})\to \pi_1(X',\overline{z'})\to \GL(H^2_{et}(Y'_{\overline{z'}},\bQ_l(1)))$$
    coincides where $Y':=Y\times_X X'$. Let $S_2:=p_1(\{x'\in |X'|: x \text{ is not } l\text{-strictly Galois generic}\})$
    and put $S:=S_1\cup S_2.$ By definition, $S_1$ is sparse and by \cite[Fact 1.7.1.2.]{ambrosi2021specializationneronseverigroupspositive} and \cite[Prop. 8.5(c)]{Maulik-Poonen2012}, $S_2$ is sparse. By \cite[Prop. 8.5(b)]{Maulik-Poonen2012}, $S$ is also a sparse subset of $|X|$.

    Let $z\in X(k)\setminus S$ and $z'=p_1^{-1}(z)$. To show that $z'$ is geometrically irreducible over $k_0'$, it remains to show that $k_0'$ is algebraically closed in $k(z')$. Since $k(z)=k$,
    $$d\geq[k(z'):k(z)]\geq[k(z')\cap \bF:k(z)\cap \bF]\geq [k(z')\cap \bF:k_0'].$$
    By the definition of $k_0''$, the algebraic closure of $k_0'$ in $k(z')$ is contained in $k_0''$. So, it suffices to show that $k_0''\otimes_{k_0'} k(z')$ is a field. This follows from the definition of $S_1$.

\end{proof}
\begin{rem}
Note that if $X\subseteq\bP^1_k$ is open and dense and $k$ is Hilbertian, $X(k)\setminus S$ is an infinite set. 
\end{rem}

By the theory of Lefschetz pencil (cf. \cite[\S. 4.2]{ambrosi2021specializationneronseverigroupspositive}), there is a diagram 
$$Y \stackrel{\pi} \longleftarrow \widetilde{Y} \stackrel{f}\longrightarrow \check{\bP}^1_k$$
where $\pi$ is a blowing up at smooth $B\subseteq Y$. For any $z\in \check{\bP}^1_k(k)$, the fiber $f^{-1}(z)$ is identified, via $\pi$, with a hyperplane section $H_z$ on $Y$. Let $X$ be an open dense subset of $\check{\bP}^1_k$ such that the restriction of $f$ on $f^{-1}(X)$ is smooth. Let $g: Y\to \Spec(k)$ be the structure morphism, $\eta$ be the generic point of $X$ and $\overline{\eta}$ is a geometric point over $\eta$ (taking separable closure).
By \cite[Fact 3.3]{CHT2017annmath}, there is a connected finite etale Galois covering map $p_1:X'\to X$ such that the Zariski closure of the image of $\pi_1(X'_{\overline{k_0'}},\overline{\eta'})$ in
$\GL(H^2_{et}(\widetilde{Y}_{\overline{\eta'}},\bQ_l(1)))$ is connected, where $\eta'$ is the generic point of $X'$ and $k_0':=k(X')\cap \bF$. By an abuse of notation, we use $f: \widetilde{Y} \lra X$ to denote the base change of $f$ to $X$. Consider the base change

$$\begin{CD}
    \widetilde{Y}'@>f'>> X'\\
    @V p_2 VV @VV p_1 V\\
    \widetilde{Y} @>f >> X
\end{CD}$$

Let $z\in X(k)$ not in $S$ as in Lemma \ref{lemmaweaklef1} and $z'=p_1^{-1}(z)$. For primes $l'\neq p$ and $p$, we have the following exact sequences
\begin{equation}\label{formula1}\begin{CD}
    0@>>> \NS(\widetilde{Y}_{\eta})\otimes \bQ_{l'} @>>> H^2(\widetilde{Y}_{\overline{\eta}},\bQ_{l'}(1))^{\pi_1(X,\overline{\eta})} @>>> V_{l'}(\Br(\widetilde{Y}_{\overline{\eta}})^{G_\eta})@>>> 0
\end{CD}\end{equation}
\begin{equation}\label{formula2}\begin{CD}
    0@>>> \NS(\widetilde{Y}_{z})\otimes \bQ_{l'} @>>> H^2(\widetilde{Y}_{\overline{z}},\bQ_{l'}(1))^{G_k} @>>> V_{l'}(\Br(\widetilde{Y}_{\overline{z}})^{G_k})@>>> 0\\ 
    @. @AAA @AAA @AAA @.\\
    0@>>> \NS(Y)\otimes \bQ_{l'} @>>> H^2(Y_{k^s},\bQ_{l'}(1))^{G_k} @>>> V_{l'}(\Br(Y_{k^s})^{G_k})@>>> 0
\end{CD}\end{equation}
\begin{equation}\label{formula3}\begin{CD}
    0@>>> \NS(\widetilde{Y}_{\eta})\otimes \bQ_p @>>> H^0_{\fppf}(\eta,R^2f_*\bQ_p(1)) @>>> V_p(\Br(\widetilde{Y}_{\overline{\eta}})^{G_\eta})@>>> 0
\end{CD}\end{equation}
\begin{equation}\label{formula4}\begin{CD}
    0@>>> \NS(\widetilde{Y}_{z})\otimes \bQ_p @>>> H^0_{\fppf}(z,R^2f_{z*}\bQ_p(1)) @>>> V_p(\Br(\widetilde{Y}_{\overline{z}})^{G_k})@>>> 0\\
    @. @AAA @AAA @AAA @.\\
    0@>>> \NS(Y)\otimes \bQ_p @>>> H^0_{\fppf}(\Spec(k),R^2g_{*}\bQ_p(1)) @>>> V_p(\Br(Y_{k^s})^{G_k})@>>> 0
\end{CD}\end{equation}
where we simplify 
$$H^0_{\fppf}(-,R^2f_{*}\bQ_p(1)):=(\varprojlim_{n} H^0_{\fppf}(-,R^2f_{*}\mu_{p^n}))[1/p].$$ The vertical maps in diagrams \eqref{formula2}, \eqref{formula4} are induced by identifying $\widetilde{Y}_z$ with $H_z\subseteq Y$ as a hyperplane and $z$ with $\Spec k$. Rows of diagrams \eqref{formula3}, \eqref{formula4} are exact because of \eqref{firstexact}.\\

By \cite[Thm 1.4.2.2, Lem 4.2.1]{ambrosi2021specializationneronseverigroupspositive},  the morphisms 
\begin{equation}\label{formulaNS}
\NS(Y)\otimes \bQ\to \NS(\widetilde{Y}_\eta)\otimes \bQ\stackrel{sp^{ar}_{\eta,z}}\longrightarrow \NS(\widetilde{Y}_z)\otimes \bQ .
\end{equation}
are both isomorphisms.
 So the left vertical maps of \eqref{formula2}, \eqref{formula4} are bijective and by the weak Lefschetz theorem for etale cohomology and for crystalline cohomology (cf. Lemma \ref{lemweaklefcrys}), the middle vertical maps of these two diagrams are injective. Hence, to show that the third columns are isomorphisms, it suffices to prove that $\QQ_\ell$ (resp. $\QQ_p$ ) linear spaces in the middle of each diagram have the same dimension.
 \begin{lem}\label{lemweaklefcrys}
     Let $k$ be a finitely generated field over $\bF_p$ and $f:X\to \Spec(k)$ be a smooth projective variety over $k$ of dimension $\geq 3$. Assume that $H\subseteq X$ is a smooth hyperplane section and $f_H:H\to \Spec (k)$ is the structure map. Then the induced map 
     $$H^0_{\fppf}(\Spec(k),R^2f_{*}\bQ_p(1))\longrightarrow H^0_{\fppf}(\Spec(k),R^2f_{H*}\bQ_p(1))$$
     is injective.
 \end{lem}
 \begin{proof}
 By \cite[Lem. 3.3]{DAd}, $H^0_{\fppf}(\Spec(k),R^2f_{*}\bQ_p(1))\lra H^0_{\fppf}(\Spec(\bar{k}),R^2f_{*}\bQ_p(1))$
 is injective.  By Proposition \ref{theoremtowardscrystalline}, it suffices to show that the map 
     $$H^2_{\crys}(\overline{X}/\overline{W})_{\bQ_p}\to H^2_{\crys}(\overline{H}/\overline{W})_{\bQ_p}$$
     is injective. This follows from the weak Lefschetz theorem for crystalline cohomology (cf. \cite[Sec. 3.8(a)]{illusie1975report}).
 \end{proof}
 
By spreading out
$$\widetilde{Y}\stackrel{f}\longrightarrow X$$
over $\bF_p$, we get
$$\widetilde{\cY} \stackrel{f}\longrightarrow  \cX.$$ Similarly, by spreading out the point $z\in X$, we get a commutative diagram
$$\begin{CD}
    \widetilde{\cY}_Z@>f_{\cY_Z}>> Z\\
    @Vi_{\cY_Z}VV @V i_Z VV\\
    \widetilde{\cY}@>f>> \cX
\end{CD}$$
By suitably shrinking $Z$ and $\cX$, we can assume that $f$ is smooth, and both $Z$ and $\cX$ are smooth over $\bF_p$.
By Theorem \ref{thekeythm} 
                $$\dim H^2_{\fppf}(\eta,R^2f_*\bQ_p(1))= \dim H^0_{\rig}(\cX,R^2f_*\cO_{\cY}^\dagger(1))^{F=1}$$
$$\dim H^2_{\fppf}(z,R^2f_{z*}\bQ_p(1))= \dim H^0_{\rig}(Z,R^2f_{\cY_Z*}\cO_{\cY_Z}^\dagger(1))^{F=1}$$
Note that $R^2f_{\cY_Z*}\cO_{\cY_Z}^\dagger=i_Z^*R^2f_*\cO_{\cY}^\dagger$ and the functor $i_Z^*: (F\text{-})\isocd(\cX)\to (F\text{-})\isocd(Z)$ is faithful because restricting to a closed point is a faithful functor. 

\begin{prop}\label{proprational}
    The right vertical morphisms in \eqref{formula2} and \eqref{formula4}
    $$V_{l'}(\Br(Y_{k^s})^{G_k})\longrightarrow V_{l'}(\Br(\widetilde{Y}_{\overline{z}})^{G_k}) $$
    are isomorphisms for all prime $l'$.
\end{prop}

\begin{proof}
Firstly, we show that it suffices to prove these two equalities below
    \begin{equation}\label{rationall-part}
        \dim H^2(\widetilde{Y}_{\overline{z}},\bQ_{l'}(1))^{G_k}=\dim H^2(\widetilde{Y}_{\overline{\eta}},\bQ_{l'}(1))^{\pi_1(X,\overline{\eta})}. 
    \end{equation}
   \begin{equation}\label{rationalp-part2}
        \dim H^0_{\rig}(\cX,R^2f_*\cO_{\cY}^\dagger(1))^{F=1}= \dim H^0_{\rig}(Z,R^2f_{\cY_Z*}\cO_{\cY_Z}^\dagger(1))^{F=1}.
    \end{equation}
Assume that these two equalities hold. Combine them with \eqref{formulaNS}, we have 
    $$\dim V_{l'}(\Br(\widetilde{Y}_{\overline{z}})^{G_k})=\dim V_{l'}(\Br(\widetilde{Y}_{\overline{\eta}})^{G_\eta}) $$
    for all primes $l'$. But then we have 
    $$\dim V_{l'}(\Br(Y_{k^s})^{G_k})= \dim V_{l'}(\Br(\widetilde{Y}_{k^s})^{G_k})\geq  \dim V_{l'}(\Br(\widetilde{Y}_{\overline{\eta}})^{G_\eta})$$
    where the first equality is because Brauer group is birational invariant and the second one is Corollary \ref{fgfield}.

    We first prove the case $l'\neq p$, that is \eqref{rationall-part}. Recall that the action of $\pi_1(z')$ and $\pi_1(X')$ factors as the following diagram 
    \begin{equation}\label{formulafactor}\xymatrix{
    \pi_1(z') \ar[r] \ar@{->>}[d] & \pi_1(X') \ar[r] \ar@{->>}[d]& \GL(H^2(\widetilde{Y}_{\overline{\eta}},\bQ_{l'}(1)))=\GL(H^2(\widetilde{\cY}_{\overline{t}},\bQ_{l'}(1)))\\
    \pi_1(Z',\overline{t})\ar[r] & \pi_1(\cX',\overline{t}) \ar[ur]
    }\end{equation}
    where $\overline{t}$ is a geometric point of a closed point $t'$ of $Z'$ and can be regarded as a geometric point of a closed point $t$ of $Z$.
    Therefore
    \begin{equation}\label{pointtomodel}\begin{aligned}
    \dim H^0(Z'_{\bF},R^2f'_{Z*}\bQ_{l'}(1))& = \dim H^2(\widetilde{Y}'_{\overline{z}'},\bQ_{l'}(1))^{\pi_1(Z'_\bF,\overline{t})}\\
    \dim H^0(\cX'_{\bF},R^2f'_{*}\bQ_{l'}(1)) &= 
    \dim H^2(\widetilde{Y}'_{\overline{\eta}},\bQ_{l'}(1))^{\pi_1(\cX'_{\bF},\overline{\eta}')}
    \end{aligned}\end{equation}
    By \cite[Fact 2.2.1.2.]{ambrosi2021specializationneronseverigroupspositive}, the left hand side of \eqref{pointtomodel} is independent of $l'$. For our chosen prime $l$, the geometric monodromy groups of $\pi_1(Z'_\bF,\overline{t})$ and $\pi_1(\cX'_\bF,\overline{\eta}')$ acting on $H^2(\widetilde{\cY}_{\overline{t}},\bQ_{l}(1))$ coincide (cf. \cite[\S 2.2.2]{ambrosi2021specializationneronseverigroupspositive}). So, for all primes $l'\neq p$,
    \begin{equation}\label{formulageosame}
    H^2(\widetilde{Y}'_{\overline{z}'},\bQ_{l'}(1))^{\pi_1(Z'_\bF,\overline{t})}\cong H^2(\widetilde{Y}'_{\overline{\eta}},\bQ_{l'}(1))^{\pi_1(\cX'_{\bF},\overline{\eta}')}. 
     \end{equation}
    Note that $\widetilde{Y}'_{\overline{\eta}}\cong \widetilde{Y}_{\overline{\eta}}$ and $\widetilde{Y}'_{\overline{z}'}\cong \widetilde{Y}_{\overline{z}}$. Also, by the same diagram as \eqref{formulafactor}, the action of $\pi_1(z'_\bF)$ and $\pi_1(X'_{\bF})$ on $H^2(\widetilde{Y}'_{\overline{z}'},\bQ_{l'}(1))$ factors through $\pi_1(Z'_\bF,\overline{t})$ and $\pi_1(\cX'_{\bF},\overline{\eta}')$, respectively. Once a path from $\overline{\eta}$ to $\overline{z}$ is chosen, there is an identification (since $R^2f_*\bQ_{l'}(1)$ is a local system) $H^2(\widetilde{Y}_{\overline{z}},\bQ_{l'}(1))\cong H^2(\widetilde{Y}_{\overline{\eta}},\bQ_{l'}(1))$ which is compatible with the actions of $G_k$ and $\pi_1(X,\overline{\eta})$ on the left and righd hand sides respectively. Let $G$ be the Galois group of the covering $X'\to X$. We have exact sequences
    $$0\longrightarrow \pi_1(X')\longrightarrow \pi_1(X)\longrightarrow G
    \longrightarrow 0$$
    $$0\longrightarrow \pi_1(X'_\bF)\longrightarrow \pi_1(X')\longrightarrow G_{k_0'}\longrightarrow 0.$$
    There are similar exact sequences for $Z$, replacing $X'$ and $X'_\bF$ by $Z'$ and $Z'_\bF$ respectively. The desired result \eqref{rationall-part} follows from taking $G_{k_0'}$-invaraint and further taking $G$-invariant part of both sides of \eqref{formulageosame}.

    Now we prove the $p$-part: since the pullback to a closed point is faithful, the natural map
    $$H^0_{\rig}(\cX,R^2f_*\cO_{\cY}^\dagger(1))\lra H^0_{\rig}(Z,R^2f_{\cY_Z*}\cO_{\cY_Z}^\dagger(1))$$
    is injective.Hence, to show \eqref{rationalp-part2}, it suffices to show
    \begin{equation}\label{rationalp-part3}
        \dim H^0_{\rig}(\cX,R^2f_*\cO_{\cY}^\dagger(1))= \dim H^0_{\rig}(Z,R^2f_{\cY_Z*}\cO_{\cY_Z}^\dagger(1)).
    \end{equation}
    By an abuse of notation, write $\cX'_{\bF}$ (resp. $Z'_{\bF}$) for the base change $\cX'\otimes_{k_0}{\bF}$ (resp. $Z'\otimes_{k_0}\bF$), where $k_0=K(X)\cap \bF$. The natural map
    \begin{equation}\label{rationalp-part4}
    H^0(\cX'_{\bF},R^2f'_*\bQ_l(1))\cong H^0(Z'_{\bF},R^2f'_{Z*}\bQ_l(1))
    \end{equation}
    is an isomorphism. \eqref{rationalp-part3} follows from the computation
    $$\begin{aligned}
        &\dim H^0_{\rig}(\cX,R^2f_*\cO_{\cY}^\dagger(1))=\dim H^0(\cX_{\bF},R^2f_*\bQ_l(1))= \dim H^0(\cX'_{\bF},R^2f'_*\bQ_l(1))^G\\
        &=\dim H^0(Z'_{\bF},R^2f'_{Z*}\bQ_l(1))^G= \dim H^0(Z_{\bF},R^2f_{Z*}\bQ_l(1))=\dim H^0_{\rig}(Z,R^2f_{Z*}\cO_{\cY_Z}^\dagger(1))
    \end{aligned}$$
    The first and last equalities are because \cite[Fact 2.2.1.2.]{ambrosi2021specializationneronseverigroupspositive}, the second and forth are by Galois descent and the third one follows from taking $G$-invariant in \eqref{rationalp-part4}.
\end{proof}

\subsection{Integral result}
 Consider $l'$ large enough such that the sequence \eqref{formulaexact2} is exact for $Y,\widetilde{Y}_{\eta}$ and $\widetilde{Y}_{z}$. We have a morphism of exact sequences
\begin{equation}\label{integralexact3}\begin{CD}
    0@>>>  \NS(Y)/l'^n\NS(Y)@>>> H^2(Y_{k^s},\mu_{l'^n})^{G_k}@>>> \Br(Y_{k^s})^{G_k}[l'^n]@>>> 0\\
    @. @VVV @VVV @VVV @.\\
    0@>>>  \NS(\widetilde{Y}_z)/l'^n\NS(\widetilde{Y}_z)@>>> H^2(\widetilde{Y}_{\overline{z}},\mu_{l'^n})^{G_k} @>>> \Br(\widetilde{Y}_{\overline{z}})^{G_k}[l'^n]@>>> 0.
\end{CD}\end{equation}

 \begin{prop}\label{propintegral}
     The right vertical map $\gamma:\Br(Y_{k^s})^{G_k}[l'^n]\to \Br(\widetilde{Y}_{\overline{z}})^{G_k}[l'^n]$ of \eqref{integralexact3} is an isomorphism for primes $l'$ sufficiently large.
 \end{prop}

 \begin{proof}
  Note that $\NS(Y)$ and $\NS(\widetilde{Y}_{\overline{z}})$ are finitely generated abelian groups and the map $\NS(Y)\otimes \bQ\to \NS(\widetilde{Y}_{\overline{z}})\otimes \bQ$ is an isomorphism. So, after only considering larger $l'$ if necessary, we can assume the left vertical map of \eqref{integralexact3} is an isomorphism. By weak Lefschetz theorem, the middle vertical map of \eqref{integralexact3} is injective thus the right vertical map is injective. 
  Since both the domain and target are finite abelian groups ($l$-part of Brauer groups is of cofinite type when $l\neq p$), it suffices to compare the cardinality. By the weak Lefschetz theorem, the pullback $H^2(Y_{k^s}, \mu_{l'^n}) \lra H^2(\widetilde{Y}_{\overline{z}},\mu_{l'^n})$
  is injective. Consequently, the pullback $H^2(Y_{k^s}, \mu_{l'^n}) \lra H^2(\widetilde{Y}_{\overline{\eta}}, \mu_{l'^n})$ is also injective. For sufficiently large $l'$, the pullback
  $$\NS(Y)/l'^n \lra \NS(\widetilde{Y}_\eta)/l'^n $$
  is an isomorphism for any $n\geq 1$. Hence, for sufficiently large $l'$, the pullback
  $$\Br(Y_{k^s})^{G_k}[l'^\infty]\lra\Br(\widetilde{Y}_{\overline{\eta}})^{G_{\eta}}[l'^\infty]$$
  is injective.
  By Corollary \ref{fgfield}, for sufficiently large $l'$,
  $\Br(Y_{k^s})^{G_k}[l'^n]\cong \Br(\widetilde{Y}_{\overline{\eta}})^{G_{\eta}}[l'^n]$ for any $n\geq 1$. Considering the exact sequence
$$\begin{CD}
    0@>>>  \NS(\widetilde{Y}_\eta)/l'^n\NS(\widetilde{Y}_\eta)@>>> H^2(\widetilde{Y}_{\overline{\eta}},\mu_{l'^n})^{G_\eta}@>>> \Br(\widetilde{Y}_{\overline{\eta}})^{G_\eta}[l'^n]@>>> 0
\end{CD}$$
and use the same argument on the Neron-Severi groups by applying \cite[Lem 4.2.1]{ambrosi2021specializationneronseverigroupspositive}, it suffices to show there is an isomorphism
\begin{equation}\label{integraltarget}
 H^2(\widetilde{Y}_{\overline{\eta}},\mu_{l'^n})^{G_\eta}\cong H^2(\widetilde{Y}_{\overline{z}},\mu_{l'^n})^{G_k}.\end{equation}

 Since $R^2f'_*\mu_{l'^n}$ is a local system over $X$, the specialization map gives an isomorphism 
 \begin{equation}\label{isomorphismlocsys}
     H^2(\widetilde{Y}'_{\overline{z}},\mu_{l'^n})\stackrel{\cong}\longrightarrow H^2(\widetilde{Y}'_{\overline{\eta}},\mu_{l'^n}).
    \end{equation}
 By a theorem of Gabber \cite{Gabbertorsion1983}, we may assume $H^2(\widetilde{Y}'_{\overline{\eta}},\bZ_{l'}(1))$ and $H^2(\widetilde{Y}'_{\overline{z}},\bZ_{l'}(1))$ are torsion free. In this case, we have 
 $$H^2(\widetilde{Y}'_{\overline{\eta}},\bZ_{l'}(1))^{G_{K(X')\bF}}=H^2(\widetilde{Y}'_{\overline{\eta}},\bZ_{l'}(1))\cap H^2(\widetilde{Y}'_{\overline{\eta}},\bQ_{l'}(1))^{G_{K(X')\bF}}$$
 $$H^2(\widetilde{Y}'_{\overline{z}},\bZ_{l'}(1))^{G_{k\bF}}=H^2(\widetilde{Y}'_{\overline{z}},\bZ_{l'}(1))\cap H^2(\widetilde{Y}'_{\overline{z}},\bQ_{l'}(1))^{G_{k\bF}}.$$
 By the proof of \eqref{rationall-part} (cf. \eqref{formulageosame} and the action of $G_{K(X')\bF}$ factors through $\pi_1(X'_{\bF})$), we have an isomorphism induced by specilization map
 $$H^2(\widetilde{Y}'_{\overline{\eta}},\bQ_{l'}(1))^{G_{K(X')\bF}}\cong H^2(\widetilde{Y}'_{\overline{z}},\bQ_{l'}(1))^{G_{k\bF}}.$$
 Thus, the specialization map induces an isomorphism
 $$H^2(\widetilde{Y}'_{\overline{\eta}},\bZ_{l'}(1))^{G_{K(X')\bF}}\cong H^2(\widetilde{Y}'_{\overline{z}},\bZ_{l'}(1))^{G_{k\bF}}$$
 and by Lemma \ref{lemmalast}, an isomorphism
 $$H^2(\widetilde{Y}'_{\overline{\eta}},\mu_{l^n})^{G_{K(X')\bF}}\cong H^2(\widetilde{Y}'_{\overline{z}},\mu_{l^n})^{G_{k\bF}}.$$
 for all $n$. Since $\widetilde{Y}'_{\overline{\eta}}\cong \widetilde{Y}_{\overline{\eta}}$ and $\widetilde{Y}'_{\overline{z}'}\cong \widetilde{Y}_{\overline{z}}$, taking $G$-invariant, we get the desired isomorphism \eqref{integraltarget}.
\end{proof}

\begin{lem}\label{lemmalast}
    Let $k_0$ be a finitely generated field over $\bF_p$ and $X/k_0$ be a smooth projective geometrically integral variety. Put $k:=k_0\bF\subseteq (k_0)^s$ and $\Pi:=\Gal(k^s/k)$. Then for primes $l'$ sufficiently large, there is a canonical isomorphism
    $$H^2(X_{k^s},\bZ_{l'}(1))^{\Pi}\otimes \bZ/l^n\cong H^2(X_{k^s},\mu_{l'^n})^{\Pi}.$$
\end{lem}

\begin{proof}
    Let $M:=H^2(X_{k^s},\bZ_{l'}(1))$. By a theorem of Gabber \cite{Gabbertorsion1983}, $M$ is a finite free $\bZ_{l'}$-module when $l'$ large enough. So we have an exact sequence 
    $0\to M\stackrel{l'^n}\to  M\to M/l'^nM\to 0.$
    Taking $\Pi$-invariant part, the conclusion follows from that $H^1(\Pi,M)[\ell^\prime]=0$ (cf. \cite[\S 4.3 and Thm. 4.5]{CHT2017annmath}).
\end{proof}

\appendix

\section{$p$-adic Tate via convergent cohomology}\label{appendixA}

In this appendix, we reprove a result of P\'al (\cite[Prop. 6.9]{Pal}) via Ogus' formalism of convergent F-isocrystals. This is because it is not clear to us how to simultaneously establish the Leray spectral sequence and identify higher direct image with Ogus’ one after completion.\\

Let $L$ be a finitely generated field in characteristic $p$ which is the function field of a smooth quasi-projective variety $S$ of dimension $s$ over $\bF_p$. Assume $X$ is a smooth projective variety geometrically irreducible of dimension $d$ over $L$. By shrinking $S$ if necessary, we assume $f:\sX\to S$ is a smooth projective morphism between two smooth varieties over $\bF_p$ such that the generic fiber $\sX_\eta$ is isomorphic to $X$ as an $L$-scheme. We call such $f:\sX\to S$ a model of $X$ over $L$. Let $K:=\bQ_p$. We consider Ogus's convergent cohomology \cite{ogus1984f} and put 
$$\cH^i(\sX):=H^0(S,R^if_{\Ogus *}\cO_{\sX/K}).$$
Here $R^if_{\Ogus *}\cO_{\sX/K}$ is the convergent $F$-isocrystal defined by Ogus (\cite[Thm. 3.1]{ogus1984f}). Note that it is a finite dimensional $\bQ_p$-vector space. This is because $H^0(S,R^if_{\Ogus *}\cO_{\sX/K})=\Hom_{\isoc}(\cO_S,R^if_{\Ogus *}\cO_{\sX/K})$ and pullback to a closed point is a faithful functor between  isocrystals (cf. \cite[Thm. 4.1]{ogus1984f}) as $S$ is irreducible. The usual multiplication map $\cO_{\sX/K}\otimes \cO_{\sX/K}\to \cO_{\sX/K}$ induces a cup product 
$$R^if_{\Ogus*}\cO_{\sX/K}\otimes R^jf_{\Ogus*}\cO_{\sX/K}\to R^{i+j}f_{\Ogus*}\cO_{\sX/K} $$
which induces a cup product $\cH^i(\sX)\otimes\cH^j(\sX)\to \cH^{i+j}(\sX)$. By restricting to Frobenius eigenspaces, it induces a cup product $H^i(\sX)^{F=p^i}\otimes H^j(\sX)^{F=p^j}\to \cH^{i+j}(\sX)^{F=p^{i+j}}$.

Suppose $z\in Z_r(X)$ is a cycle and shrink $S$ if necessary, we can assume there is $z'\in Z_{r+s}(\sX)$ whose restriction to the generic fiber is $z$. Let $\gamma_{\sX}(-): CH^r(\sX)\to H^{2r}_{\crys}(\sX)$ be the crystalline cycle class map (cf. \cite{gros1985classes}) where $CH^r(\sX)$ is the rational equivalence classes of cycles in codimension $r$.
Let $\eta_X(z)$ be the image of $\gamma_{\sX}(z')_K$ under the composition of the edge map of the spectral sequence
$$ H^{2d-2r}_{\crys}(\sX)_K\to H^0_{\crys}(S,R^{2d-2r}f_{\crys*}\cO_{\sX/\bZ_p})_K$$
and the isomorphism $H^0_{\crys}(S,R^{2d-2r}f_{\crys*}\cO_{\sX/\bZ_p})_K\cong H^0(S,R^{2d-2r}f_{\Ogus *}\cO_{\sX/K})$ (cf. \cite[Coro. A.13]{Mor}).
 Let $\eta^r_\sX: Z_{d-r}(X)\to \cH^{2r}(\sX)$  be the map constructed above. 
 
 \begin{rem}
 	The image of crystalline cycle class map $CH^r(\sX)_K\to H^{2r}_{\crys}(\sX)_K$ lies in the Frobenius eigenspace $H^{2r}_{\crys}(\sX)^{F=p^r}_K$. The edge map of the spectral sequence preserves Frobenius action, thus the image of $\eta^r_{\sX}$ lie in $\cH^{2r}(\sX)^{F=p^r}\subseteq \cH^{2r}(\sX)$.
 \end{rem}
 
 \begin{lem}\label{lemmacanonical}
 	\begin{enumerate}
 		\item The space $\cH^{2r}(\sX)^{F=p^r}$ does not depend on the model $f:\sX\to S$ and $X\to \cH^{2r}(X)^{F=p^r}$ is a contravariant functor.
 		\item The cup product 
 		$$\cH^{2i}(\sX)^{F=p^i}\otimes \cH^{2j}(\sX)^{F=p^j}\to \cH^{2i+2j}(\sX)^{F=p^{i+j}}$$
 		is independent of the model $f:\sX\to S$.
 		\item The cycle class $\eta_X^r\in \cH^{2r}(X)^{F=p^r}$ does not depend on the model $f:\sX\to S$ and the choice of $z'\in Z_{s+r}(\sX)$.
 	\end{enumerate}
 \end{lem}
 
 \begin{proof}
 	The proof is the same as the proof of \cite[Prop. 5.11, Lem. 5.13, 5.17]{Pal} with small changes. For convenience of the reader, we summarize it as the following.\par 
 	For (1), we consider the same category $\cI$ constructed in the proof of \cite[Prop. 5.11]{Pal}. Let 
 	$$(\alpha,\beta): (U,\sX,\pi,\phi,\iota)\longrightarrow (U',\sX',\pi',\phi',\iota')$$ 
 	be a morphism between two objects in $\cI$. By fully faithfulness of restriction functor $\Fisoc(U')\longrightarrow \Fisoc(U)$, $(\alpha,\beta)$ induces an isomorphism 
 	$$\cH^r(\alpha,\beta): H^0(U',R^{2r}\pi'_*(\cO_{\sX'/K}))^{F=p^r}\longrightarrow H^0(U,R^{2r}\pi_*(\cO_{\sX/K}))^{F=p^r}.$$
 	Then Pal proved $\cI$ is a filtered category (\cite[Lem. 5.12]{Pal}) thus the morphism 
 	$$H^0(U,R^{2r}\pi_*(\cO_{\sX/K}))^{F=p^r}\longrightarrow \lim_{\cI} H^0(U',R^{2r}\pi'_*(\cO_{\sX'/K}))^{F=p^r}$$
 	is an isomorphism. The first claim follows. For the second claim, given a map $f:X\to Y$ between smooth, projective varieties over $L$ and $g:\sX\to \sY$ is a morphism between models over $U$ whose generic point is $L$. The morphism $F$ induces a morphism between convergent $F$-isocrystals $R^r\pi_{\sX*}(\cO_{\sX/K})\longrightarrow R^r\pi_{\sY*}(\cO_{\sY/K})$ over $U$. Apply $\Hom_{\Fisoc}(\cO(-r),-)$ to this morphism, we get desired morphism and the first claim shows it does not depend on choice of $\sX,\sY,U$.\par 
 	For (2), it is short and identical to the proof of \cite[Lem. 5.13]{Pal}.\par 
 	For (3), we consider the same category $\cI(z)$ constructed in the proof of \cite[Prop. 5.17]{Pal}. It is a filtered cateogry (\cite[Lem. 5.18]{Pal}). Since the proof of (1) shows that 
 	$$H^0(U,R^{2r}\pi_*(\cO_{\sX/K}))^{F=p^r}\longrightarrow \lim_{\cI(z)} H^0(U',R^{2r}\pi'_*(\cO_{\sX'/K}))^{F=p^r}$$
 	is an isomorphism, the claim follows.
 \end{proof}
 
 By lemma above, we will replace the notation $\cH^{2r}(\sX)^{F=p^r}$ by $\cH^{2r}(X)^{F=p^r}$ to emphasise it is independent of the model. We denote $\eta_X^r: Z_r(X)\to \cH^{2r}(X)^{F=p^r}$ the map that factors $\eta_{\sX}^r$. The map $\eta_X^r$ is independent of the model.
 
 \begin{rem}
 	Since $\eta_{\sX}^r$ is the composition of $\eta_X^r$ and $\cH^{2r}(X)^{F=p^r}\to \cH^{2r}(\sX)$, it is well-defined, i.e, independent of the choice of the cycle $z'$.
 \end{rem}
 
 \begin{rem}
 	We learned from D'Addezio that restriction to an open dense subset of $S$ induces an isomorphism between $\cH^r(\sX)$ (cf. \cite[Thm 3.2.4]{daddezio2022heckeorbitsshimuravarieties}). Thus Lemma \ref{lemmacanonical} also applies to $\cH^r(\sX)$.
 \end{rem} 
 
 \begin{lem}
 	The morphism $\eta^r_X$ factors through $CH^r(X)$. In particular, $\eta_{\sX}^r$ factors through $CH^r(X)$ for any model $\sX$.
 \end{lem}
 
 \begin{proof}
 	Assume $z\in Z_{d-r}(X)$ can be written as a divisor of a rational function $\phi$ on a irreducible subvariety $Y\subseteq X$ of codimension $r-1$. We can take Zariski closure of $Y$ in $\sX$ and view $\phi$ as a rational function on it. So $Y$ is the restriction of a cycle on $\sX$ which is rational equivalent to $0$. The result follows as the crystalline cycle class map factors through $CH^r(\sX)$.
 \end{proof}
 
 \begin{defn}
	Let $T(X,r,p)$ be the Tate conjecture which says that the cycle class map
	$$\eta_X^r: CH^r(X)_K\to \cH^{2r}(X)^{F=p^r}$$
	is surjective. For a prime number $l\neq p$, let $T(X,r,l)$ be the usual Tate conjecture which says that the cycle class map
$$\eta_{X,l}^r: CH^r(X)_{\bQ_l}\to H^{2r}(X_{L^s},\bQ_l(r))^{G_L}$$
is surjective.
\end{defn}

We want to show the following result:
\begin{prop}\label{propositionindependence} The following statements are equivalent:
    \begin{enumerate}
        \item The claim $T(X,1,l)$ is true for some prime number $l$.
        \item The claim $T(X,1,l)$ is true for all prime numbers $l$.
    \end{enumerate}
\end{prop}
The argument follows P\'al.
 
 \begin{prop}
 	The morphism $\eta^i_X$ factors through $CH^i(X)\to CH^i_{SN}(X)$ where $CH^i_{SN}(X)$ is the quotient of $CH^i(X)$ by the subgroup $SNCH^i(X)$ generated by the rational equivalence classes of cycles smash-nilpotent to zero.
 \end{prop}
 
 \begin{proof}
  The proof is the same as \cite[Prop. 5.23]{Pal}. For convenience of the reader, we sketch it as the following.\par 
 	Let $\rho_i:\sX_i\to S$ ($i=1,...,n$) be models of $X$ over $L$ and $\rho: \sX_1\times_S \cdots \times_S \sX_n\to S$ be the fiber product of $\rho_i$. Denote $\pi_i:\sX_1\times_S \cdots \times_S \sX_n\to \sX_i$ the $i$-th projection. Let $\rho_{i_1,\cdots,i_n*}$ be the map 
 	$$R^{i_1}\rho_{1*}\cO_{\sX_1/K}\otimes\cdots \otimes R^{i_n}\rho_{n*}\cO_{\sX_n/K}\to R^{i_1+\cdots +i_n}\rho_*\cO_{\sX_1\times \cdots\sX_n/K}$$
 	of convergent $F$-isocrystals induced by fiberwise exterior cup product. We define 
 	$$\cH^{i_1,...,i_n}(\sX_1,...,\sX_n/K):=H^0(S,R^{i_1}\rho_{1*}\cO_{\sX_1/K}\otimes\cdots \otimes R^{i_n}\rho_{n*}\cO_{\sX_n/K})$$
 	and the map 
 	$$\overline{\pi}_{i_1,...,i_n}^*: \cH^{i_1,...,i_n}(\sX_1,...,\sX_n/K)\longrightarrow \cH^{i_1+\cdots +i_n}(\sX_1\times_S\cdots\times_S\sX_n/K) $$
 	to be the map induced by $ \rho_{i_1,\cdots,i_n*}$. Let $\pi_{i_1,...,i_n}^*$ be the map
 	$$\cH^{i_1}(\sX_1/K)\otimes_K \cdots \otimes_K \cH^{i_n}(\sX_n/K)\longrightarrow \cH^{i_1+\cdots+i_n}(\sX_1\times_S\cdots \times \sX_n/K) $$
 	given by $x_1\otimes \cdots \otimes x_n\mapsto \pi_1^*(x_1)\cup \cdots \cup \pi_n^*(x_n)$.
 	 By proper base change and Kunneth formula, we see that
 	$$\bigoplus_{i_1+\cdots +i_n=r} \rho_{i_1,\cdots,i_n*}: \bigoplus_{i_1+\cdots +i_n=r} \cH^{i_1,...,i_n}(\sX_1,...,\sX_n/K)\longrightarrow \cH^{i_1+\cdots +i_n}(\sX_1\times_S\cdots\times_S\sX_n/K)$$ 
 	is an isomorphism which implies $\overline{\pi}_{i_1,...,i_n}^*$ is also an isomorphism. Since $\pi_{i_1,...,i_n}^*$ is the composition of the injective tensor product 
 	(one can see it via lifting $S$ and viewing convergent $F$-isocrystals as vector bundles with connections) 
 	$$\cH^{i_1}(\sX_1/K)\otimes_K \cdots \otimes_K \cH^{i_n}(\sX_n/K)\longrightarrow \cH^{i_1,...,i_n}(\sX_1,...,\sX_n/K)$$
 	and $\overline{\pi}_{i_1,...,i_n}^*$, $\pi_{i_1,...,i_n}^*$ is also injective. In particular, the induced map 
 	\begin{equation}\label{formulainjective}
 	\cH^{i_1}(X/K)^{F=p^{i_1}}\otimes_K \cdots \otimes_K \cH^{i_n}(X/K)^{F=p^{i_n}}\longrightarrow \cH^{i_1+\cdots+i_n}(X^n/K)^{F=p^{i_1+...+i_n}}
 	\end{equation}
 	is injective. 
 	
 	Let $z\in Z^r(X)$ be a smash-nilpotent cycle, we need to see $\eta^r_X(z)=0$. By definition there is a positive integer $n$ such that the $n$-fold smash product $z^{\otimes n}$ on the $n$-fold direct product $X^n$ is rationally equivalent to zero. Thus 
 	$$0=\eta^{nr}_{X^n}(z^{\otimes n})= p_1^*(z)\cup \cdots\cup p_n^*(z)\in \cH^{2rn}(X^n)^{F=p^{rn}}$$
 	where $p_i:X^n\to X$ is the $i$-th projection. By injectivity of \eqref{formulainjective}, the map
 	$$\cH^{2r}(X/K)^{F=p^{r}}\otimes_K \cdots \otimes_K \cH^{2r}(X/K)^{F=p^{r}}\longrightarrow \cH^{2nr}(X^n/K)^{F=p^{nr}}$$
 	given by $z_1\otimes\cdots\otimes z_n\mapsto p_1^*(z_1)\cup\cdots\cup p_n^*(z_n)$ is injective. The claim follows.
 	 \end{proof}

   By a theorem of Voevodsky (\cite{voevodsky1995nilpotence}), $ACH^r(X)\subseteq SNCH^r(X)$ where $ACH^r(X)$ is the rational equivalence classes of cycles algebraic equivalent to $0$. The proposition above implies $\eta_X^r$ factors through $CH^r_A(X)$, the algebraic equivalence classes of cycles.
 	 
\begin{lem}\label{lemmacommute}
	There exists a model $f:\sX\to S$ of $X$ such that there exists a morphism $\beta: \cH^{2d}(\sX)\to K$ making the following diagram commute
	$$\begin{CD}
		CH^0_A(X)@>\eta_{\sX}^d>> \cH^{2d}(\sX)\\
		@V\deg VV @VV\beta V\\
		\bZ@>>> K.
	\end{CD}$$
\end{lem}
\begin{proof}
	We first choose a model $\sX\to S$. Since $CH^0_A(X)$ is finitely generated, use generic flatness and shrink $S$ if necessory, we can assume for each $z\in CH^0_A(X)$ there is a cycle $z'\in \sX$ whose intersection with $X$ represents $z$ such that for each $s\in |S|$ the fiber $\sX_s$ intersects with $z'$ properly. Fix $s\in |S|$ with residue field $k$ and put $K':=W(k)$. The map $\beta$ is defined as the composition 
$$H^0(S, R^{2d}f_{\Ogus*}\cO_{\sX/K})\to (R^{2d}f_{\Ogus*}\cO_{\sX/K})_s\stackrel{\cong} \longrightarrow H^{2d}(\sX_s,\cO_{X_s/K})\stackrel{Tr_{\sX_s}}\longrightarrow K$$
where the first map is restriction, the second map is proper base change (\cite[Remark 3.1.7]{ogus1984f}). The third one is the composition of the usual trace map $H^{2d}(\sX_s,\cO_{\sX/K'})\to K'$ and $\deg(K'/K)^{-1}Tr_{K'/K}(-)$.\par 

To see the commutativity, let $z\in CH^0(X)$ which has degree $a$ and $z'$ as above. Since $\sX_s$ intersects $z'$ properly, the fiber $z'_s$ is $0$-dimensional with degree $a$. Thus it is enough to show the commutativity in the absolute case. That is, suppose $X$ is smooth projective irreducible of dimension $d$ over a perfect field $k$ in characteristic $p$. We need to see commutativity of 
$$\xymatrix{
	CH^0(X) \ar_{\gamma^0_X}[rd] \ar^{deg}[r] & K\\
	 & H_{\conv}^{2d}(X)\ar_{Tr_X}[u]
	}$$
	Assume $x\in X$ has degree $a$. The compatibility of the Gysin map with the trace map (cf. \cite[Prop. VII 2.3.1]{berthelot2006cohomologie}) gives the following commutative diagram
	$$\xymatrix{
	H^0(x)\ar_{Tr_x}[rd] \ar^{G_{x/X}}[rr] & & H^{2n}(X)\ar^{Tr_X}[ld]\\
	& K &}$$ 
	Then $\deg(x)=Tr_x(1)=Tr_XG_{x/X}(1)$ as desired.

\end{proof}

\begin{lem}\label{oguspairing}
There exists a model $f:\sX\to S$ of $X$ such that there exists a commutative diagram:

$$\begin{CD}
	CH^1_A(X)_K\times CH_A^{d-1}(X)_K @>\langle , \rangle>> K\\
	@V\eta^1_\sX\times \eta^{d-1}_\sX  VV @AA\beta A\\
	\cH^{2}(\sX)\times \cH^{2d-2}(\sX) @>\cup >> \cH^{2d}(\sX)
\end{CD}
$$
 The map $\langle ,\rangle$ is induced by the intersection product together with $deg:CH^{d}(X)\to \bZ$. 
\end{lem}

\begin{proof}
We can find $\sX$ and $\beta$ as in Lemma \ref{lemmacommute}. We know that the cycle class map is compatible with intersection product (cf. \cite[VI Coro. 4.3.15]{berthelot2006cohomologie}) thus we have a commutative diagram
	$$\begin{CD}
		CH^1(X)_K\times CH^{d-1}(X)_K @>\langle ,\rangle >> CH^0(X)_K\\
		@V\eta^1_\sX\times \eta^{d-1}_\sX  VV @VV\eta^0_\sX V\\
		\cH^{2}(\sX)\times \cH^{2d-2}(\sX) @>\cup >> \cH^{2d}(\sX).
	\end{CD}$$ 
	It suffices to show the commutativity of 
	$$\xymatrix{
	CH^0(X) \ar_{\eta^0_\sX}[rd] \ar^{deg}[r] & K\\
	 & \cH^{2d}(\sX)\ar_{\beta}[u]
	}$$
	which is Lemma \ref{lemmacommute}.
\end{proof}

Assume $\sX,S$ as in the above lemma. By a result due to Ambrosi (\cite[Thm. 6.5.4.1]{ambrosi2021specializationneronseverigroupspositive}), there is an overconvergent $F$-isocrystal $\cE^{i,\dagger}$ whose completion is $R^if_{\Ogus*}\cO_{\sX/K}$. Note that 
$$H^0_{\rig}(S,\cE^{i,\dagger})^{F=p^r}\cong \Hom_{\Fisocd}(\cO^\dagger_{S}(-r),\cE^{i,\dagger})\cong \Hom_{\Fisoc}(\cO_S(-r),R^if_{\Ogus*}\cO_{\sX/K}) \cong \cH^i(X)^{F=p^r}$$
and there is a canonical inclusion
$$H^0_{rig}(S,\cE^{i,\dagger})\cong\Hom_{\isocd}(\cO^\dagger_{S},\cE^{i,\dagger})\subseteq \Hom_{\isoc}(\cO_S,R^if_{\Ogus*}\cO_{\sX/K}) \cong \cH^i(X).$$
Denote $H^0_{rig}(S,\cE^{i,\dagger})$ by $\cH^i_{\rig}(\sX)$. $\cH^i_{\rig}(\sX)^{F=p^r}$ is identified with $\cH^i(X)^{F=p^r}$ under the inclusion above thus is independent of $\sX$.

Let $\cL$ be a lisse $\bQ_l$-sheaf on $U$ when $l\neq p$ and be an overconvergent $F$-isocrystal over $U$ when $l=p$. We define the $L$-function to be 
$$L(U,\cL,t):=\prod_{x\in |U|} \det(1-t^{\deg(x)}\cdot Frob_x(\cL))^{-1}.$$
When $l\neq p$, we put $H^n(\sX)_l:=R^nf_{*}(\bQ_{l,\sX})$ to be the lisse $\bQ_l$ sheaf on $S$ and when $l=p$ we put $H^n(\sX)_p:=\cE^{n,\dagger}$ to be the overconvergent $F$-isocrystal over $S$.

\begin{lem}\label{pallem}
	We have $L(S,H^n(\sX)_l,t)\in \bQ[[t]]$ and are equal for different $l$.
\end{lem}

\begin{proof}
	It is enough to prove that for each $x\in |U|$, the factor 
	$$\det(1-t^{\deg(x)}\cdot Frob_x(\cL))^{-1}$$
	has coefficient in $\bQ$ and are equal for different $l$. By proper smooth base change \cite[Remark 3.7.1]{ogus1984f} , it is reduced to the absolute case. In this case there is an equivalence of categories $\Fisoc\simeq \Fisocd$ and the result follows from main result of \cite{katz1974some}.
	\end{proof}
	Let $\rho$ be the common order of the pole of the $L$-function at $t=p^{-d-s+1}$. The following proposition implies proposition \ref{propositionindependence}. 

\begin{prop}\label{propconvergentformulation}
	For different $l$, the following are equivalent:
\begin{enumerate}
	\item The claim $T(X,1,l)$ holds
	\item The rank of $NS(X)$ is $\rho$.
\end{enumerate}
\end{prop}

\begin{proof}
	It is well-known when $l\neq p$. For the case $l=p$, we follow argument of Pal. By Matsusaka's theorem on equivalence of algebraic equivalence and numerical equivalence for divisors, the map 
	$\eta_X^1: CH^1_{A}(X)_K\to \cH^2(X)^{F=p}$ is injective. Let $W \subseteq \cH^2_{\rig}(\sX)$ be the generalised eigenspace of $F$ with eigenvalue $p$, that is, the union $\cup_{i=1}^\infty \ker(F-p)^i$. Let $F$ act on $CH^1_{A}(X)_K $ as multiplication by $p$. By Lemma \ref{oguspairing}, the injection map $\eta_X^1: CH^1_{A}(X)_K\to \cH^2(\sX)$ is split as $\bQ_p$-linear spaces with the $F$-action. This is because we can restrict the pairing between cohomology groups to $\cH^2(\sX)\times V':=Im(\eta_X^{d-1})\to K$. Note that the $F$-action on $V'$ is multipication by $p^{d-1}$ and the pairing $CH^1_A(X)_K\times W'\to K$ induced by $\eta^1_X$ is left nondegenerated (Matsusaka's theorem). We can find a $F$-equivariant subspace $V\subseteq V'$ such that the pairing $CH^1_A(X)_K\times V\to K$ is perfect and $\cH^2(\sX)\cong CH^1_A(X)_K\oplus V^\perp$. It follows that the injective map $CH^1_{A}(X)_K\to W$ is also split as $\bQ_p$-linear spaces with the $F$-action. We will conclude from the following lemma: 
 \begin{lem}
		$\dim_K W=\rho$.
	\end{lem}
	\begin{proof}
	We have the Etesse-Le Stum trace formula (cf. \cite[Thm. 6.3]{etesse1993fonctions})
	$$L(S,\cE^{2d-2,\dagger},t)=\prod_{i=0}^{2m}\det(1-tF|H^i_{rig,c}(S,\cE^{2d-2,\dagger}))^{(-1)^{i+1}}$$
	In fact, $\cE^{2d-2,\dagger}$ is pure of weight $2d-2$ as purity is defined pointwise, thus, by \cite[Thm. 5.3.2]{kedlaya2006fourier}, for $i<2m$ the cohomology $H^i_{rig,c}(S,\cE^{2d-2,\dagger})$ is mixed of weight $\leq 2d-2+i<2(d+m-1)$. Therefore, only the $2m$-th factor contribute to the order of the pole at $t=q^{-d-m+1}$. Let $\alpha_i$ are the reciprocal roots of the characteristic polynomial of $F$ acting on $H^{2m}_{rig,c}(S,\cE^{2d-2,\dagger})$. Note that the category of overconvergent $F$-isocrystals form a Tannakian category (cf. \cite{crew1992f}), in particular, it is closed under dual. Since in the category of $F$-isocrystals, $R^2f_{\Ogus*}\cO_{\sX/K}$ is dual to $R^{2d-2}f_{\Ogus*}\cO_{\sX/K}$ and passing to $F$-isocrystals is a fully faithful functor, we get $\cE^{2,\dagger}\cong \underline{\Hom}_{\Fisocd}(\cE^{2d-2,\dagger},\cO^{\dagger}_S(d))$. The Poincare duality of rigid cohomology (cf. \cite[Thm. 1.2.3]{kedlaya2006finiteness}) tells us that the reciprocal roots of the characteristic polynomial of $F$ acting on $H^0_{\rig}(S,\cE^{2,\dagger})$ are $q^{d+s}/\alpha_i$. Thus the order of vanishing of the $L$-function at $t=q^{-d-m+1}$ is the dimension of the generalized eigenspace of $F-p$.
	\end{proof}
	Since $\cH^2(X)^{F=p}\subseteq W$, $\dim_K\cH^2(X)^{F=p}\leq \rho$. If $(b)$ holds, the injective map $\eta_X^1$ must be surjective. Assume $(a)$ is true, then $\cH^2(X)^{F=p}\subseteq W$ is a direct summund respecting action of $F$ and contains all eigenvectors of $F$. It implies $\cH^2(X)^{F=p}=W$. 
	 \end{proof}

\section{The rigid Leray spectral sequence via the overconvergent site (By Veronika Ertl)} \label{appendixB}

Let $k$ be a perfect field of characteristic $p>0$, $W(k)$ its ring of Witt vectors and $K$ its fraction field. 
These data define a good overconvergent variety 
$(k,K):=(\Spec(k)\hookrightarrow\Spf(W(k))\leftarrow\Spa(K))$ \cite[Defs.\,2.2.1, 2.5.1]{LeStum_2011}. 
To an algebraic $k$-variety $X$, one can associate the good overconvergent site $(X/K)^\dagger$  (see \cite[Defs.\,2.4.11, 2.5.13]{LeStum_2011} and the paragraphs after them), 
equipped with a structure sheaf $\sO^\dagger_{X/K}$ given by the restriction of the structure sheaf of $(W(k)/K)^\dagger$ \cite[Cor.\,3.3.3]{LeStum_2011}. 
The objects in this category are overconvergent varieties over $X$. 
These are diagrams $(Y \hookrightarrow \fY \xleftarrow{\ssp} \fY_K \xleftarrow{\lambda} \fQ)/X$, where $Y \hookrightarrow \fY$ is a formal embedding over $W(k)$, $\ssp:\fY_K\rightarrow \fY$ is the specialisation map from the generic fibre and $\lambda: \fQ\rightarrow \fY_K$ is a morphism of analytic varieties over $K$, together with a morphism $Y\rightarrow X$. 

\begin{example}
The notion of overconvergent varieties is a generalisation of the notion of frames which is usually used in the calculus of rigid cohomology. 
Recall that a triple $(Y,\overline{Y},\fY)$ is a $W(k)$-frame, if $Y\hookrightarrow \overline{Y}$ is an open immersion of $k$-varieties and $\overline{Y} \hookrightarrow \fY$ is a closed immersion into a formal $W(k)$-scheme. 
One can consider the tubes $]Y[_{\fY}$ and $]\overline{Y}[_{\fY}$ of $Y$ and $\overline{Y}$ inside $\fY$ \cite[Def.\,1.3]{berth-rig}. 

A $W(k)$-frame is called proper, if $\overline{Y}$ is proper. 
In that case, on the one hand by \cite[Ex.\,3.5]{ertl2024berthelotsconjecturehomotopytheory} the strict inclusion of analytic varieties $]Y[_{\fY}\Subset ]\overline{Y}[_{\fY}$ defines a dagger structure on $]Y[_{\fY}$ \cite[Def.\,3.3]{LBV}. 
On the other hand, $(Y \hookrightarrow \fY \xleftarrow{\ssp} \fY_K \leftarrow ]\overline{Y}[_{\frY})$ is an overconvergent variety. 
Moreover, according to \cite[Cor.\,2.3.15]{LeStum_2011} in the overconvergent site $(W(k)/K)^\dagger$ any overconvergent variety is isomorphic to one of this form. 

More generally, we can consider $W(k)$-frames $(Y,\overline{Y},\fY)/X$ over $X$, meaning, they come equipped with a morphism $Y\rightarrow X$. 
Similar to above, a proper $W(k)$-frame $(Y,\overline{Y},\fY)/X$ over $X$ gives rise to an overconvergent variety $(Y \hookrightarrow \fY \xleftarrow{\ssp} \fY_K \leftarrow ]\overline{Y}[_{\fY})/X$ over $X$, and up to isomorphism, any overconvergent variety over $X$ is of this form. 

Thus up to isomorphism, there is an equivalence between overconvergent varieties in  $(X/K)^\dagger$ and proper $W(k)$-frames over $X$. 
\end{example}

\begin{rem}
Note that in this note, we only use the notion of the ``good'' overconvergent site. 
So we omit the subscript $g$ that appears in \cite{LeStum_2011}, also in the different coefficient categories.
\end{rem}

\begin{rem}
Sometimes a sheaf on $(X/K)^\dagger$ is called an overconvergent sheaf on $X/K$, 
and an $\sO^\dagger_{X/K}$-module is called an overconvergent module on $X/K$ etc.\,\cite[Def.\,3.3.4]{LeStum_2011}.
\end{rem}

In \cite[Defs.\,3.3.4, 3.3.7, Prop.\,3.3.11]{LeStum_2011} Le Stum introduces a category of coefficients on $(X/K)^\dagger$ which correspond to the category of overconvergent isocrystals on $X/K$ \cite[Def.\,8.1.3]{lestumbook} for rigid cohomology. We recall both here.  

\begin{defn}
For an algebraic $k$-variety $X$, denote by $\Crys^\dagger(X/K)$ the category of crystals on $(X/K)^\dagger$, that is the category of $\sO^\dagger_{X/K}$-modules where all the transition maps are isomorphisms. 
Let $\Mod^\dagger_{\ffp}(X/K)$ be the full sub category of $\Crys^\dagger(X/K)$ consisting of finitely presented modules. 
Recall that $\sE\in\Crys^\dagger(X/K)$ is finitely presented if every realisation 
$\sE_{\frQ}$ is a coherent $i_Y^{-1}\sO_{\frQ}$-module \cite[Prop.\,3.3.11]{LeStum_2011} with $i_Y:]Y[_{\frQ}\hookrightarrow \frQ$ for an overconvergent variety $(Y \hookrightarrow \frY \xleftarrow{\ssp} \frY_K \xleftarrow{\lambda} \frQ)/X$. 
\end{defn}

\begin{defn}
An overconvergent isocrystal over $X/K$ consists of the following data:
\begin{enumerate}
    \item 
    For every morphism of $W(k)$-frame $(Y,\overline{Y},\frY)$ over $X$ 
    a coherent $j^\dagger_{Y}\mathcal{O}_{]\overline{Y}[_{\frY}}$-module $E_{\frY}$.
    \item 
    For every morphism of $W(k)$-frames $f:(Y',\overline{Y}',\frY')\rightarrow (Y,\overline{Y},\frY)$
    over $X$ an isomorphism
    $$
    \phi_f: f^\ast E_{\frY} \cong E_{\frY'}
    $$
    satisfying the cocycle condition.
\end{enumerate}
We denote by $\isoc^\dagger(X/K)$ the category of finitely presented overconvergent isocrystals on $X/K$.
\end{defn}

\begin{thm}
There is a canonical equivalence of categories 
$$
\Mod^\dagger_{\ffp}(X/K) \cong \isoc^\dagger(X/K).
$$
\end{thm}
\begin{proof}
This is \cite[Thm.\,4.6.7]{LeStum_2011}. 
It is possible to describe the functor explicitly. 
Recall that any sheaf $\sF$ on $(X/K)^\dagger$ is uniquely determined by its realisations $\sF_{\frQ}$ \cite[Def.\,3.1.8]{LeStum_2011} on every overconvergent variety $(Y\hookrightarrow\frY\leftarrow\frY_K \leftarrow \frQ)$ over $X$ subject to the cocycle condition \cite[Prop.\,3.1.9]{LeStum_2011}. 
Each realisation is a sheaf on $]Y[_{\frQ}$. 
But by \cite[Cor.\,2.3.15]{LeStum_2011}, every overconvergent variety is equivalent to a $W(k)$-frame $(Y,\overline{Y},\frY)$. 
Thus $\sF$ on $(X/K)^\dagger$ is equivalent to a family of sheaves on the tubes $]\overline{Y}[_{\frY}$ subject to the cocycle condition. 
Now if $\sF$ is in $\Mod^\dagger_{\ffp}(X/K)\subset \Crys^\dagger(X/K)$ the transition maps are isomorphisms.  
Moreover it is finitely presented. This means, by \cite[Cor.\,3.2.13]{LeStum_2011} $\sE\in\Mod^\dagger_{\ffp}(X/K)$ that for every proper frame $(Y,\overline{Y},\frY)$, the realisation $\sE_{]\overline{Y}[_{\frY}}:= \sE_{\frY}$ is a coherent $j^\dagger_{Y}\mathcal{O}_{]\overline{Y}[_{\frY}}$-module. 
Thus the functor is given by sending $\sE$ to all its realisations on proper frames. 
\end{proof}

According to \cite[Cor.\,4.6.8]{LeStum_2011}, this allows us to compute the rigid cohomology of an overconvergent isocrystal $\sE$ on  $X/K$:
\begin{cor}
For $\sE\in\isoc^\dagger(X/K)$, there is a canonical isomorphism
$$
H^\ast_\rig(X/K,\sE)\cong H^\ast((X/K)^\dagger,\sE). 
$$ 
In particular, for the structure sheaf which corresponds to the trivial isocrystal, 
one has
$$
H^\ast_\rig(X/K)=H^\ast((X/K)^\dagger,\sO^\dagger_{X/K}).
$$
\end{cor}

Assume  now given a morphism of algebraic varieties $\rho:X\rightarrow  S$. 
As explained in the paragraph after \cite[Def.\,2.4.11]{LeStum_2011}, 
it induces a morphism of overconvergent sites 
$\rho: (X/K)^\dagger\rightarrow  (S/K)^\dagger$.  
Now Grothendieck's formalism provides a Leray spectral sequence \cite[\href{https://stacks.math.columbia.edu/tag/0731}{Lem.\,0730}]{stacks-project}
\begin{equation}\label{equ: Leray1}
E^{i,j}_2=H^j((S/K)^\dagger, R^i\rho_\ast\sO^\dagger_{X/K}) \Longrightarrow 
H^{i+j}((X/K)^\dagger,\sO^\dagger_{X/K}).
\end{equation}
As we have seen, we can identify the abutment of this spectral sequence with the rigid cohomology $H_\rig^{i+j}(X/K)$. 
In case that $\rho$ is proper and smooth we can compute the $E_2$-sheet in terms of rigid cohomology, too. 

\begin{prop}\label{ertlberthelot}
Let $\rho:X\rightarrow S$ be a smooth and proper morphism of algebraic varieties. 
For any $n\in\bN$, $R^n\rho_\ast\sO^\dagger_{X/K}$ is a finitely presented overconvergent module on $(S/K)^\dagger$. 
In particular, there is an overconvergent isocrystal $R^n\rho_{\rig\ast}(X/S) \in \isoc^\dagger(S/K)$ which corresponds to $R^n\rho_\ast\sO^\dagger_{X/K}$ under the equivalence 
$$
\Mod^\dagger_{\ffp}(S/K) \cong \isoc^\dagger(S/K).
$$
\end{prop}

\begin{proof}
According to the second part of \cite[Thm.\,4.6.7]{LeStum_2011}, the realisation of $R^n\rho_\ast\sO^\dagger_{X/K}$ on a proper frame $(T,\overline{T},\frT)$ over $S$ is given by Berthelot's rigid push-forward $R^n\rho_{\rig,\ast}(X/(T,\overline{T},\frT))$ \cite[Rem.\,(2.5)c)]{berth-rig} (compare \cite[\S1.1]{ertl2024berthelotsconjecturehomotopytheory}). 
But by \cite[Cor.\,3.37]{ertl2024berthelotsconjecturehomotopytheory} this is the realisation of a canonical overconvergent isocrystal over $S$.
\end{proof}

\begin{cor}\label{appendspectral}
Let $\rho:X\rightarrow  S$ be a proper smooth morphism of algebraic varieties. 
There is a spectral sequence
$$
E^{i,j}_2=H^j_{\rig}(S/K,R^i\rho_{\rig\ast}(X/S)) \Rightarrow H_\rig^{i+j}(X/K)
$$
of finite dimensional $K$-vector spaces compatible with Frobenius action. 
\end{cor}

\begin{proof}
It only remains to note that $R^i\rho_{\rig\ast}(X/S)$ is naturally an $F$-isocrystal, then functoriality implies the claim.
\end{proof}

\begin{rem}
The canonical functor 
$$
\Fisocd(S/K) \rightarrow \Fisoc(S/K)
$$
which is fully faithful by \cite{kedlaya2004full} maps the overconvergent $F$-isocrystal $R^n\rho_{\rig\ast}(X/S)$ to Ogus convergent $F$-isocrystal $R^n\rho_{\Ogus\ast}(X/S)$ (see \cite[Corollary 2.34]{shiho_relative1}). 
In particular,  $R^n\rho_{\rig\ast}(X/S)$ coincides with  the overconvergent $F$-isocrystal obtained by Ambrosi \cite{ambrosi2021specializationneronseverigroupspositive}. 
Note that the latter is constructed in a way that does not allow to identify all its realisations on proper $W(k)$-frames over $S$ explicitly (compare \cite[Rem.\,4.9]{shiho_relative1}). 
Thus a priori, it cannot be identified  with $R^n\rho_\ast\sO^\dagger_{X/K}$ under the equivalence $\Mod^\dagger_{\ffp}(S/K) \cong \isoc^\dagger(S/K)$, but with the above argumentation this becomes clear. 
\end{rem}

\section* {Acknowledgements}
We would like to thank Marco D'Addezio and Alexie N.Skorobogatov for helpful comments, questions and discussions on the first version of the paper. The first named author would like to thank Matthew Morrow for important communications on $F$-crystals and log-convergent site; to the second named author for bringing attention to this problem.  The second-named author would like to express gratitude to Xinyi Yuan for introducing him to the questions on Brauer groups and for his guidance. Additionally, the second-named author extends thanks to Veronika Ertl and Timo Keller for important discussions. Veronika Ertl and Timo Keller were in the early stage of this project.

\bibliographystyle{plain}
\bibliography{ptorsionbib}

   \Addresses
\end{document}